\documentclass[a4paper]{amsart}

\usepackage{amsmath,amsthm,amssymb,enumerate}
\usepackage{epsfig}
\usepackage{amssymb}
\usepackage{amsmath}
\usepackage{amssymb}
\usepackage{amsmath,amsthm}
\usepackage[latin1]{inputenc}
\usepackage[T1]{fontenc}
\usepackage{path}
\usepackage{ae,aecompl}
\usepackage{amsfonts}
\usepackage{amsxtra}
\usepackage{euscript,mathrsfs}
\usepackage{color}
\usepackage[left=3.4cm,right=3.4cm,top=4cm,bottom=4cm]{geometry}
\usepackage[citecolor=blue,colorlinks=true]{hyperref}
\allowdisplaybreaks

\usepackage{enumitem}
\setenumerate{label={\rm (\alph{*})}}

\usepackage{amsfonts}
\usepackage{amsxtra}

\usepackage[frak,hyperref,theorem_section]{paper_diening}
\numberwithin{equation}{section}

\newcommand{\Div}{\divergence}

\newcommand{\R}{\mathbb R}

\newcommand{\E}{\mathbb E}
\newcommand{\p}{\mathbb P}

\newcommand{\F}{\mathfrak F}

\newcommand{\pid}{\pi_{\mathrm{det}}}

\newcommand{\dd}{\mathrm d}
\newcommand{\dx}{\, \mathrm{d}x}

\newcommand{\ds}{\, \mathrm{d}\sigma}
\newcommand{\dt}{\, \mathrm{d}t}
\newcommand{\dxt}{\,\mathrm{d}x\, \mathrm{d}t}
\newcommand{\dxs}{\,\mathrm{d}x\, \mathrm{d}\sigma}
\newcommand{\dif}{\mathrm{d}}

\newcommand{\mf}{\mathfrak{F}}

\newcommand{\prst}{\mathbb{P}}

\newcommand{\mt}{\mathbb{T}^2}
\newcommand{\tor}{\mathbb{T}^2}

\DeclareMathOperator{\diver}{div}

\begin{document}

\title[Convergence rates for stochastic Navier--Stokes equations]{Convergence rates for the numerical approximation of the 2D stochastic Navier--Stokes equations}

\author{Dominic Breit \& Alan Dodgson}
\address{Department of Mathematics, Heriot-Watt University, Riccarton Edinburgh EH14 4AS, UK}
\email{d.breit@hw.ac.uk, ad335@hw.ac.uk}

%
%

\begin{abstract}
We study stochastic Navier--Stokes equations in two dimensions with respect to periodic boundary conditions. The equations are perturbed by a nonlinear multiplicative stochastic forcing with linear growth (in the velocity) driven by a cylindrical Wiener process.
 We establish convergence rates for a finite-element based space-time approximation with respect to convergence in probability (where the error is measure in the $L^\infty_tL^2_x\cap L^2_tW^{1,2}_x$-norm). Our main result provides linear convergence in space
and convergence of order (almost) 1/2 in time. This improves earlier results from [E.~Carelli, A.~Prohl: Rates of convergence for discretizations of the stochastic incompressible Navier--Stokes equations. SIAM J. Numer. Anal. 50(5), 2467--2496. (2012)] where the convergence rate in time is only (almost) 1/4.
Our approach is based on a careful analysis of the pressure function using a stochastic pressure decomposition.

\end{abstract}


\date{\today}

\maketitle

%
%
%
%
%
%
%
%
%
%

\section{Introduction}

In this paper we are concerned with the stochastic Navier--Stokes equations
\begin{align}\label{eq:SNS}
\left\{\begin{array}{rc}
\dd\bfu=\mu\Delta\bfu\dt-(\nabla\bfu)\bfu\dt-\nabla\pi\dt+\Phi(\bfu)\dd W
& \mbox{in $Q$,}\\
\Div \bfu=0\qquad\qquad\qquad\qquad\qquad\,\,\,\,& \mbox{in $Q$,}\\
\bfu(0)=\bfu_0\,\qquad\qquad\qquad\qquad\qquad&\mbox{ \,in $\mathcal O$,}\end{array}\right.
\end{align}
on a filtered probability space $(\Omega,\F,(\F_t)_{t\geq0},\prst)$. The equations are perturbed by an $(\F_t)$-Wiener process (possibly infinite dimensional) and $\Phi$ grows linearly in $\bfu$ (see Section 2.1 for the precise assumptions). The quantity $\mu>0$ is the viscosity of the fluid and $\bfu_0$ is a given (random initial datum).
Here the unknowns are the velocity field $\bfu:\Omega\times Q\rightarrow\R^N$ and the pressure $\pi:\Omega\times Q  \rightarrow\R$, where $Q=(0,T)\times\mathcal O$ and $\mathcal O\subset\R^N$ with $N=2,3$.\\
The stochastic perturbation in the balance of momentum \eqref{eq:SNS}$_1$ can take into account for physical, empirical or numerical uncertainties and thermodynamical fluctuations. In addition to that, a main reason why the stochastic Navier--Stokes equations \eqref{eq:SNS} became so popular in fluid mechanical research is their application to turbulence theory (see, for instance, \cite{Bi} and \cite{MikRoz}). Its mathematical investigation started in the 70's with the pioneering paper of Bensoussan and Temam \cite{BeTe}. They provide a semi-deterministic approach based on the flow-transformation. A first fully stochastic theory has been developed by Flandoli and Gatarek \cite{FlaGat} by showing the existence of a martingale solution. These solutions are weak in the stochastic sense meaning that
the underlying probability space as well the Wiener process $W$ are not a priori known but become an integral part of the solution. In two dimensions, when uniqueness is known, a stochastically strong solution exists (it is defined on a given probability space with a given Wiener process), see \cite{CC}.
Nowadays there is a huge amount of literature concerning the analysis of \eqref{eq:SNS} and most of the results from the deterministic theory found their stochastic counterpart. For an overview we refer to the recent survey article
\cite{Ro}.\\
The situation about the numerical approximation of \eqref{eq:SNS} is totally different and only very few results are available.
In \cite{BCP} a fully practical space-time approximation for the three-dimensional stochastic Navier--Stokes equations in a bounded domain is studied. It is shown that the sequence of approximate solutions converges in law (up to a subsequence) to a martingale solution if both discretization parameters tend to zero. This is the best result one can hope for without using some unproven hypotheses about the space-regularity of solutions
(or to be content with local-in-time results). In two dimensions the situation is much better, at least if periodic boundary conditions are considered, that is
\[
\mathcal O = \mathbb T^2 = \left( (-\pi, \pi)|_{\{ - \pi, \pi \}} \right)^2.
\]
The space-regularity of the unique strong solution is well-known (see for instance \cite{KukShi}). Based on this the convergence rates for a finite-element based space-time approximation
is analysed in \cite{CP}. The result is linear convergence in space
and convergence of order (almost) 1/4 in time. The precise estimate comparing
the solution $\bfu$ and its space-time approximation $\bfu_{h,m}$ reads as
\begin{align}\label{eq:thm:4intro}
\begin{aligned}
\E\bigg[\mathbf{1}_{\Omega_{\Delta t,h}}\bigg(\max_{1\leq m\leq M}\|\bfu(t_m)-\bfu_{h,m}\|_{L^2_x}^2&+\sum_{m=1}^M \Delta t \|\nabla\bfu(t_m)-\nabla\bfu_{h,m}\|_{L^2_x}^2\bigg)\bigg]\\&\leq \,c\,\big(h^2+(\Delta t)^{2\alpha}\big)
\end{aligned}
\end{align}
for any $\alpha<1/4$. Here we have $\Omega_{\Delta t,h}\subset\Omega$ with $\p\big(\Omega\setminus\Omega_{\Delta t,h}\big)\rightarrow0$ as $\Delta t,h\rightarrow0$. Consequently, one can infer
from \eqref{eq:thm:4intro} convergence in probability (as defined by Printems \cite{Pr}) with asymptotic rates
(almost) $1/4$ and 1 respectively. The aim of the present paper is to improve
the convergence rates in time from (almost) $1/4$ to (almost) $1/2$, see Theorem \ref{thm:4} for the precise statement.
This is certainly the optimal convergence rate in time in view of the stochastic forcing and should also be considered as the natural result because of the space-regularity of the solution. 
\\
The reason for the low convergence rate in time in \cite{CP}
is the low time-regularity of the pressure function $\pi$ in \eqref{eq:SNS}. Its appearance can only be avoided when working with finite element functions which are exactly divergence-free. Unfortunately, their construction is quite complicated such that the preferred descretizations (such as the Taylor--Hood, the Crouzeix--Raviart, and the MINI element, see \cite{BF,GL,GS}) are only 
asymptotically divergence-free. The low regularity of the pressure gradient
arises from the stochastic forcing (in fact, $\nabla\pi$ behaves as $\dd W$) and it does not seem possible to improve unless the noise is divergence-free (this is quite restrictive as it only allows certain additive or linear multiplicative noise). In the general case (non-divergence-free finite elements and nonlinear multiplicative noise) a more subtle analysis of the pressure function
is required. Our main idea is a decomposition of the pressure into a deterministic and a stochastic component
(such a decomposition first appeared in \cite{Br}). The deterministic pressure part
behaves as the convective term $\bfu\otimes\bfu$. The latter one can be estimated along the lines of \cite{CP}
(following classical deterministic arguments combined with a discrete stopping time).
In addition a second stochastic integral appears which behaves similarly to the original stochastic integral. Although the time-regularity of this part has not improved, we benefit from the averaging properties of the It\^{o}-integral. Combining these ideas finally leads to the optimal convergence rate in Theorem \ref{thm:4}.\\
The paper is organized as follows. In Section 2 we present the mathematical framework, that is the probability setup, the concept of solutions and their qualitative properties. In particular, we give improved (compared to \cite{CP}) results on the time-regularity of $\nabla\bfu$, see Corollary \ref{cor:uholder} b). This is needed in Section 3 in order to estimate the error between the continuous solution and the time-discrete solution. The heart of the paper is Section 4 were we estimate the error between the time-discrete solution and the space-time discretization.
Crucial tools are the space-regularity of the time-discrete solution from \cite{BCP}
and the decomposition of the corresponding pressure function.

\section{Mathematical framework}
\label{sec:framework}

\subsection{Probability setup}

Let $(\Omega,\F,(\F_t)_{t\geq0},\prst)$ be a stochastic basis with a complete, right-continuous filtration. The process $W$ is a cylindrical Wiener process, that is, $W(t)=\sum_{k\geq1}\beta_k(t) e_k$ with $(\beta_k)_{k\geq1}$ being mutually independent real-valued standard Wiener processes relative to $(\F_t)$ and $(e_k)_{k\geq1}$ is a complete orthonormal system in a sepa\-rable Hilbert space $\mathfrak{U}$.
To give the precise definition of the diffusion coefficient $\varPhi$, consider $\bfz\in L^2(\mt)$ and let $\,\varPhi(\bfz):\mathfrak{U}\rightarrow L^2(\mt)$ be defined by $\Phi(\bfz)e_k=\bfg_k(\cdot,\bfz(\cdot))$. In particular, we suppose
that $\bfg_k\in C^1(\mt\times\R^2)$ and that the following conditions hold
\begin{align}\label{eq:phi}
\begin{aligned}
\sum_{k\geq1}|\bfg_k(x,\bfxi)|^2 \leq c(1+|\bfxi|^2)&,\qquad
\sum_{k\geq1}|\nabla_{\bfxi} \bfg_k(x,\bfxi)|^2\leq c,\\
\sum_{k\geq1}|\nabla_x \bfg_k(x,\bfxi)|^2 &\leq c(1+|\bfxi|^2),\quad x\in \mt,\,\bfxi\in\R^2.
\end{aligned}
\end{align}
If we are interested in higher regularity some further assumptions are in place and we require additionally $\bfg_k\in C^2(\mt\times\R^2)$ together with
\begin{align}\label{eq:phi2}
\begin{aligned}
\sum_{k\geq1}|\nabla_x^2\bfg_k(x,\bfxi)|^2 \leq c(1+|\bfxi|^2)&,\qquad
\sum_{k\geq1}|\nabla^2_{\bfxi} \bfg_k(x,\bfxi)|^2\leq \frac{c}{1+|\bfxi|^2},\\
\sum_{k\geq1}|\nabla_x\nabla_{\bfxi} \bfg_k(x,\bfxi)|^2 &\leq c,\quad x\in \mt,\,\bfxi\in\R^2.
\end{aligned}
\end{align}
We remark that the first inequality of \eqref{eq:phi} implies
\begin{align}\label{norm1}
\|\Phi(\bfu)\|_{L_2(\mathfrak U;L^2_x)}\leq\,c\big(1+\|\bfu\|_{L^2_x}\big)\quad\forall \bfu\in L^2(\mt),
\end{align}
all inequalities from \eqref{eq:phi} imply
\begin{align}\label{norm2}
\|\Phi(\bfu)\|_{L_2(\mathfrak U;W^{1,2}_x)}\leq\,c\big(1+\|\bfu\|_{W^{1,2}_x}\big)\quad\forall \bfu\in W^{1,2}(\mt),
\end{align}
and
\eqref{eq:phi} finally yields
\begin{align}\label{norm3}
\|\Phi(\bfu)\|_{L_2(\mathfrak U;W^{2,2}_x)}\leq\,c\big(1+\|\bfu\|_{W^{2,2}_x}\big)\quad\forall \bfu\in W^{2,2}(\mt).
\end{align}
In fact, all our results apply if we replace \eqref{eq:phi} and \eqref{eq:phi2} by the corresponding norm estimates above.\\
Furthermore, the conditions imposed on $\varPhi$, particularly the first assumption from \eqref{eq:phi}, allow us to define stochastic integrals
Given an $(\mathfrak F_t)$-progressively measurable process $\bfu\in L^2(\Omega;L^2(0,T;L^2(\mt)))$, the stochastic integral $$t\mapsto\int_0^t\varPhi(\bfu)\,\dif W$$
is a well defined process taking values in $L^2(\mt)$ (see \cite{PrZa} for the detailed construction). Moreover, we can multiply by test functions to obtain
 \begin{align*}
\bigg\langle\int_0^t \varPhi(\bfu)\,\dd W,\bfphi\bigg\rangle=\sum_{k\geq 1} \int_0^t\langle \bfg_k(\bfu),\bfphi\rangle\,\dd\beta_k,\quad \bfphi\in L^2(\mt).
\end{align*}
Similarly, we can define stochastic integrals with values in $W^{1,2}(\mt)$ and $W^{2,2}(\mt)$ respectively if $\bfu$ belongs to the corresponding class.\\
The following lemma is a helpful tool to analysis the time-regularity
of stochastic integrals (see, e.g., \cite[Lemma 9.1.3. b)]{Br2} or \cite[Lemma 4.6]{Ho1}).
\begin{lemma}
\label{lems:Holder}
 Let $\psi\in L^r(\Omega;L^r(0,T;L_2(\mathfrak U,L^2(\mt))))$, $r>2$, by an $(\mathfrak F_t)$-progressively measureable process and $W$ a cylindrical $(\F_t)$-Wiener process on $\mathfrak U$. Then the paths of the process $Z_t:=\int_0^t \psi\,\dd W$ are $\p$-a.s. H\"older continuous with exponent $\alpha\in\big(\frac{1}{r},\frac{1}{2}\big)$ and it holds
\begin{align*}
\E\Big[\|Z\|_{C^{\alpha}([0,T];L^2(\mt))}^r\Big]\leq c_\alpha\,\E\bigg[\sup_{0\leq t\leq T}\|\psi\|_{L_2(\mathfrak U,L^2(\mt))}^2\dt\bigg]^{\frac{r}{2}}.
\end{align*}
\end{lemma}


\subsection{The concept of solutions}
\label{subsec:solution}
In dimension two, pathwise uniqueness for weak solutions is known under \eqref{eq:phi}, we refer the reader for instance to Capi\'nski--Cutland \cite{CC} and Capi\'nski \cite{Ca}. Consequently, we may work with the definition of a weak pathwise solution.

\begin{definition}\label{def:inc2d}
Let $(\Omega,\mf,(\mf_t)_{t\geq0},\prst)$ be a given stochastic basis with a complete right-continuous filtration and an $(\mf_t)$-cylindrical Wiener process $W$. Let $\bfu_0$ be an $\mf_0$-measurable random variable. Then $\bfu$ is called a \emph{weak pathwise solution} \index{incompressible Navier--Stokes system!weak pathwise solution} to \eqref{eq:SNS} with the initial condition $\bfu_0$ provided
\begin{enumerate}
\item the velocity field $\bfu$ is $(\mf_t)$-adapted and
$$\bfu \in C_w([0,T];L^2_{\diver}(\mt))\cap L^2(0,T; W^{1,2}_{\diver}(\tor))\quad\text{$\p$-a.s.},$$
\item the momentum equation
\begin{align*}
&\int_{\mt}\bfu(t)\cdot\bfvarphi\dx-\int_{\mt}\bfu_0\cdot\bfvarphi\dx
\\&=\int_0^t\int_{\mt}\bfu\otimes\bfu:\nabla\bfvarphi\dx\,\dif t-\mu\int_0^t\int_{\mt}\nabla\bfu:\nabla\bfvarphi\dx\,\dif s\\
&\qquad+\int_0^t\int_{\mt}\Phi(\bfu)\cdot\bfvarphi\dx\,\dif W.
\end{align*}
holds $\p$-a.s. for all $\bfvarphi\in C^\infty_{\diver}(\mt)$ and all $t\in[0,T]$.
\end{enumerate}
\end{definition}

\begin{theorem}\label{thm:inc2d}
Let $N=2$ and assume that $\Phi$ satisfies \eqref{eq:phi}. Let $(\Omega,\mf,(\mf_t)_{t\geq0},\prst)$ be a given stochastic basis with a complete right-continuous filtration and an $(\mf_t)$-cylindrical Wiener process $W$. Let $\bfu_0$ be an $\mf_0$-measurable random variable such that $\bfu_0\in L^r(\Omega;L^2_{\mathrm{div}}(\mathbb T^2))$ for some $r>2$. Then there exists a unique weak pathwise solution to \eqref{eq:SNS} in the sense of Definition \ref{def:inc2d} with the initial condition $\bfu_0$.
\end{theorem}

 Now, for $\bfphi\in C^\infty(\mt)$ we can insert $\bfphi-\nabla\Delta^{-1}\Div\bfphi$
and obtain
\begin{align}
\nonumber
\int_{\mt}\bfu(t)\cdot\bfvarphi\dx &+\int_0^t\int_{\mt}\mu\nabla\bfu:\nabla\bfphi\dxs-\int_0^t\int_{\mt}\bfu\otimes\bfu:\nabla\bfphi\dxs\\
\label{eq:pressure}&=\int_{\mt}\bfu(0)\cdot\bfvarphi\dx
+\int_0^t\int_{\mt}\pi_{\mathrm{det}}\,\Div\bfphi\dxs
\\
\nonumber&+\int_{\mt}\int_0^t\Phi(\bfu)\,\dd W\cdot \bfvarphi\dx+\int_{\mt}\int_0^t\Phi^\pi\,\dd W\cdot \bfvarphi\dx,
\end{align}
where 
\begin{align*}
\pi_{\mathrm{det}}&=-\Delta^{-1}\Div\Div\big(\bfu\otimes\bfu\big),\\
\Phi^\pi&=-\nabla\Delta^{-1}\Div\Phi(\bfu).
\end{align*}
This corresponds to the stochastic pressure decomposition introduced in \cite{Br} (see also \cite[Chap. 3]{Br2} for a slightly different presentation). However, the situation with periodic boundary conditions we are considering here is much easier as the harmonic component of the pressure disappears. 

\subsection{Regularity of solutions}
\begin{lemma}\label{lem:reg}
Let the assumptions of Theorem \ref{thm:inc2d} be satisfied. 
 \begin{enumerate}
\item[(a)] We have
\begin{align}\label{eq:W12}
\E\bigg[\sup_{0\leq t\leq T}\int_{\mt}|\bfu|^2\dx+\int_0^T\int_{\mt}|\nabla\bfu|^2\dxt\bigg]^{\frac{r}{2}}\leq\,c_r\,\E\Big[1+\|\bfu_0\|_{L^2_x}^2\Big]^{\frac{r}{2}}.
\end{align}
\item[(b)] Assume that $\bfu_0\in L^r(\Omega,W^{1,2}_{div}(\mt))$ for some $r\geq2$. Then we have
\begin{align}\label{eq:W22}
\E\bigg[\sup_{0\leq t\leq T}\int_{\mt}|\nabla\bfu|^2\dx+\int_0^T\int_{\mt}|\nabla^2\bfu|^2\dxt\bigg]^{\frac{r}{2}}\leq\,c_r\,\E\big[1+\|\bfu_0\|_{W^{1,2}_x}^2\Big]^{\frac{r}{2}}.
\end{align}
\item[(c)] Assume that $\bfu_0\in L^r(\Omega,W^{2,2}_{div}(\mt))\cap L^{5r}(\Omega,W^{1,2}_{div}(\mt))$ for some $r\geq2$ and that \eqref{eq:phi2} holds. Then we have
\begin{align}\label{eq:W32}
\E\bigg[\sup_{0\leq t\leq T}\int_{\mt}|\nabla^2\bfu|^2\dx+\int_0^T\int_{\mt}|\nabla^3\bfu|^2\dxt\bigg]^{\frac{r}{2}}\leq\,c_r\,\E\Big[1+\|\bfu_0\|_{W^{2,2}_x}^2+\|\bfu_0\|_{W^{1,2}_x}^{10}\Big]^{\frac{r}{2}}.
\end{align}
\end{enumerate}
\end{lemma}
\begin{proof}
Part (a) is the standard a priori estimate which follows from applying It\^o's formula to the functional $f(\bfu)=\frac{1}{2}\|\bfu\|_{L^2_x}^2$ (and using Burkholder-Davis-Gundy inequality, assumption \eqref{eq:phi} and Gronwall's lemma). Note that this is legit in two dimensions since we have $\bfu\otimes\bfu\in L^2(Q)$ $\p$-a.s. by Ladyshenskaya's inequality.\\
The proof of (b) and (c) is quite similar to \cite[Corollary 2.4.13]{KukShi}. However, we are working with a different setup. So, we decided to give a formal proof although it is certainly known to experts. The proof can be made rigorous by working with a Galerkin-type approximation and show that the following estimates are uniform with respect to the dimension of the ansatz space. Such a procedure is quite standard, so we leave the details to the reader.\\
In order to show (b) we apply It\^{o}'s formula to the function $f_{\gamma}(\bfu):=\tfrac{1}{2}\|\partial_{\gamma}\bfu\|_{L^2_x}^2$ (with $\gamma\in\{1,2\}$) and obtain
\begin{align}\nonumber
\frac{1}{2}\|\partial_{\gamma}\bfu(t)\|_{L^2_x}^2&=\frac{1}{2}\|\partial_{\gamma}\bfu_0\|_{L^2_x}^2+\int_0^t f'_\gamma(\bfu)\,\dd\bfu+\frac{1}{2}\int_0^t f_\gamma''(\bfu)\,\dd\langle\langle\bfu\rangle\rangle\\
\label{eq:2Duniformest1}&=\frac{1}{2}\|\partial_{\gamma}\bfu_0\|_{L^2_x}^2+\int_{\mt}\int_0^t \partial_{\gamma}\bfu\cdot\dd\partial_{\gamma}\bfu\dx\\&+\frac{1}{2}\int_{\mt}\int_0^t\dd\Big\langle\!\Big\langle\int_0^{\cdot}\partial_{\gamma} \big(\Phi(\bfu)\,\dd W\big)\Big\rangle\!\Big\rangle\dx
=:(I)+(II)+(III).
\nonumber
\end{align}
We take the supremum in time, the $\frac{r}{2}$th power and apply expectations.
Summing over $\gamma$, we find  
\begin{align*}
(II)&=-(II)_1-(II)_2+(II)_3,\\
(II)_1&:=\mu\int_0^t\int_{\mt} |\nabla^2\bfu|^2\dxs,\\
(II)_2&:=\int_0^t\int_{\mt}\partial_{\gamma}\bfu
\cdot\partial_{\gamma}\Big(\Phi(\bfu)\,\dd W\Big)\dx,\\
(II)_{3}&:=\int_0^t\int_{\mt}(\nabla\bfu)\bfu\cdot\Delta\bfu\dxs.
\end{align*}
In two dimensions we have $(II)_3=0$ by elementary calculations. So we are left with estimating $(II)_2$ and obtain
\begin{align*}
(II)_2&=\sum_k\int_{\mt}\int_0^t\partial_\gamma\bfu\cdot\partial_\gamma\Big(\Phi(\bfu)e_k\,\dd\beta_k\Big)\dx\\
&=\sum_k\int_{\mt}\int_0^t\partial_\gamma\bfu\cdot\partial_\gamma\Big(\bfg_k(\cdot,\bfu)\,\dd\beta_k\Big)\dx\\
&=\sum_k\int_{\mt}\int_0^t \nabla_{\bfxi} \bfg_k(\cdot,\bfu)
(\partial_\gamma\bfu,\partial_\gamma\bfu)\,\dd\beta_k\dx\\
&+\sum_k\int_{\mt}\int_0^t\partial_\gamma\bfu\cdot \partial_\gamma \bfg_k(\cdot,\bfu)\,\dd\beta_k(\sigma)\dx\\
&=:(II)_2^1+(II)_2^2.
\end{align*}
On account of assumption (\ref{eq:phi}), Burkholder-Davis-Gundy inequality and Young's inequality we obtain for arbitrary $\delta>0$
\begin{align*}
\E\bigg[\sup_{0\leq t\leq T}|(II)^1_2|\bigg]^{\frac{r}{2}}&\leq \E\bigg[\sup_{0\leq t \leq T}\Big|\int_0^t\sum_k\int_{\mt}\nabla_{\bfxi} \bfg_k(\cdot,\bfu)
(\partial_\gamma\bfu,\partial_\gamma\bfu)\dx\,\dd\beta_k\Big|\bigg]^{\frac{r}{2}}\\
&\leq c\,\E\bigg[\sum_{k}\int_0^T\bigg(\int_{\mt} \nabla_{\bfxi} \bfg_k(\cdot,\bfu)
(\partial_\gamma\bfu,\partial_\gamma\bfu)\dx\bigg)^2\dt\bigg]^{\frac r4}\\
&\leq c\,\E\bigg[\bigg(\int_0^T\bigg(\int_{\mt}
|\partial_\gamma\bfu|^2\dx\bigg)^2\dt\bigg]^{\frac r4}\\
&\leq \delta\,\E\bigg[\sup_{0\leq t\leq T}\int_{\mt}
|\nabla\bfu|^2\dx\bigg]^{\frac{r}{2}}+ c(\delta)\,\E\bigg[\int_0^T\int_{\mt}
|\nabla\bfu|^2\dx\dt\bigg]^{\frac{r}{2}}\\
&\leq \delta\,\E\bigg[\sup_{0\leq t\leq T}\int_{\mt}
|\nabla\bfu|^2\dx\bigg]^{\frac{r}{2}}+c(\delta)\,\E\Big[1+\|\bfu_0\|_{L^2_x}^2\Big]^{\frac{r}{2}}
\end{align*}
using \eqref{eq:W12} in the last step.
By similar arguments we gain
\begin{align*}
\E\bigg[\sup_{0\leq t\leq T}|(II)^2_2|\bigg]^{\frac{r}{2}}
&\leq c\,\E\bigg[\int_0^T\bigg(\int_{\mt} \partial_\gamma \bfg_k(\cdot,\bfu)
\cdot\partial_\gamma\bfu\dx\bigg)^2\dt\bigg]^{\frac r4}\\
&\leq c\,\E\bigg[\bigg(\int_0^T\bigg(\int_{\mt}
|\partial_\gamma\bfu||\bfu|\dx\bigg)^2\dt\bigg]^{\frac r4}\\
&\leq c\,\E\bigg[\sup_{(0,T)}\int_{\mt}
|\bfu|^2\dx+\int_0^T\int_{\mt}
|\nabla\bfu|^2\dx\dt\bigg]^{\frac{r}{2}}\\
&\leq\,c\,\E\Big[1+\|\bfu_0\|_{L^2_x}^2\Big]^{\frac{r}{2}}.
\end{align*}
Finally, we have by (\ref{eq:phi})
\begin{align*}
(III)&=\frac{1}{2}\int_{\mt}\int_0^t \,\dd\Big\langle\Big\langle\int_0^{\cdot}\partial_\gamma \big(\Phi(\bfu)\,\dd W\big)\Big\rangle\Big\rangle\dx\\
&=\frac{1}{2}\sum_{k}\int_{\mt}\int_0^t \,\dd\Big\langle\Big\langle\int_0^{\cdot}\partial_\gamma\big(\Phi(\bfu)e_k\big) \dd\beta_k\Big\rangle\Big\rangle\dx\\
&\leq \frac{1}{2}\sum_{k}\int_0^t\int_{\mt}\big| \nabla_{\bfxi} \bfg_k(\cdot,\bfu)\partial_\gamma\bfu\big|^2\dxs\\
&+\frac{1}{2}\sum_{k}\int_0^t\int_{\mt}\big| \partial_\gamma \bfg_k(\cdot,\bfu)\big|^2\dxs\\
&\leq \,c\,\int_0^t\int_{\mt}|\nabla\bfu|^2\dxs+c\,\int_0^t\int_{\mt}|\bfu|^2\dxs.
\end{align*}
Hence we obtain by 
\eqref{eq:W12} that
\begin{align*}
\E\bigg[\sup_{0\leq t \leq T}|(III)|\bigg]^{\frac{r}{2}}&\leq\,c\,\E\Big[1+\|\bfu_0\|_{L^2_x}^2\Big]^{\frac{r}{2}}.
\end{align*}
Plugging all together and choosing $\delta$ small enough we have shown \eqref{eq:W22}.\\
The proof of (c) is similar: we simply differentiate once more.
We apply It\^{o}'s formula to the function $f^{\beta}(\bfu):=\tfrac{1}{2}\|\partial^\beta\bfu\|_{L^2_x}^2$ where $\beta\in \mathbb N_0^2$ is a multi-index of length 2. We obtain
\begin{align}\nonumber
\frac{1}{2}\|\partial^\beta\bfu(t)\|_{L^2_x}^2
&=\frac{1}{2}\|\partial^\beta\bfu_0\|_{L^2_x}^2+\int_{\mt}\int_0^t \partial^\beta\bfu\cdot\dd\partial^\beta\bfu\dx\\&+\frac{1}{2}\int_{\mt}\int_0^t\dd\Big\langle\!\Big\langle\int_0^{\cdot}\partial^\beta \big(\Phi(\bfu)\,\dd W\big)\Big\rangle\!\Big\rangle\dx
=:(I)+(II)+(III),
\nonumber
\end{align}
where
\begin{align*}
(II)&=-(II)_1+(II)_2-(II)_3,\\
(II)_1&:=\mu\int_0^t\int_{\mt} |\partial^\beta\nabla\bfu|^2\dxs,\\
(II)_2&:=\int_0^t\int_{\mt}\partial^\beta\bfu
\cdot\partial^\beta\Big(\Phi(\bfu)\,\dd W\Big)\dx,\\
(II)_{3}&:=\int_0^t\int_{\mt}\partial^\beta\big((\nabla\bfu)\bfu\big)\cdot\partial^\beta\bfu\dxs.
\end{align*}
The main difference is that $(II)_3$ does not vanish. By \cite[Lemma 2.1.20]{KukShi} (with $m=2$) and Young's inequality we have
\begin{align*}
\bigg|\int_{\mt}\partial^\beta\Div(\bfu\otimes\bfu)\cdot\partial^\beta\bfu\dx\bigg|&\leq\,c\|\bfu\|_{W^{3,2}_x}^{\frac{7}{4}}\|\bfu\|_{W^{1,2}_x}^{\frac{3}{4}}\|\bfu\|_{L^2_x}^{\frac{1}{2}}\\
&\leq\delta\|\bfu\|_{W^{3,2}_x}^2+c(\delta)\|\bfu\|_{W^{1,2}_x}^{6}\|\bfu\|_{L^2_x}^{4}\\
&\leq\delta\|\bfu\|_{W^{3,2}_x}^2+c(\delta)\big(\|\bfu\|_{W^{1,2}_x}^{10}+\|\bfu\|_{L^2_x}^{10}\big),
\end{align*}
where $\delta>0$ is arbitrary. This implies
\begin{align*}
\E\bigg[\sup_{0\leq t\leq T}|(II)_3|\bigg]^{\frac{r}{2}} &\leq\,\delta\E\bigg[\int_0^T\|\bfu\|_{W^{3,2}_x}^2\dt\bigg]^{\frac{r}{2}}+c(\delta)\E\bigg[\sup_{0\leq t\leq T}\|\bfu\|_{W^{1,2}_x}^{2}+\sup_{0\leq t\leq T}\|\bfu\|_{L^2_x}^{2}\bigg]^{\frac{5r}{2}}\\
 &\leq\,\delta\E\bigg[\int_0^T\|\bfu\|_{W^{3,2}_x}^2\dt\bigg]^{\frac{r}{2}}+c(\delta)\E\Big[\|\bfu_0\|_{W^{1,2}_x}\Big]^{\frac{5r}{2}}
\end{align*}
due to \eqref{eq:W22}.
By arguments similar to the proof of (b), using \eqref{eq:phi2}, we gain
\begin{align*}
\E&\bigg[\sup_{0\leq t\leq T}|(II)_2|\bigg]^{\frac{r}{2}}
\\&\leq\delta\,\E\bigg[\sup_{0\leq t\leq T}\|\bfu\|^2_{W_x^{2,2}}\dx\bigg]^{\frac{r}{2}}+\,c(\delta)\,\E\bigg[ \sup_{0\leq t\leq T}\|\bfu\|_{L_x^{2}}^2+\sup_{0\leq t\leq T}\|\bfu\|^2_{W_x^{1,2}}+\int_0^T\|\bfu\|_{W_x^{2,2}}^2\dt\bigg]^{\frac{r}{2}}\\
&\leq\delta\,\E\bigg[\sup_{0\leq t\leq T}\|\bfu\|_{W_x^{2,2}}^2\dx\bigg]^{\frac{r}{2}}+\,c\,\E\Big[1+\|\bfu_0\|_{W^{1,2}_x}^2\Big]^{\frac{r}{2}}
\end{align*}
using again \eqref{eq:W22}. Finally, we get
\eqref{eq:W12} that
\begin{align*}
\E\bigg[\sup_{0\leq t \leq T}|(III)|\bigg]^{\frac{r}{2}}&\leq\,c\,\E\Big[1+\|\bfu_0\|_{W^{1,2}_x}^2\Big]^{\frac{r}{2}}
\end{align*}
arguing again similarly to (b) and using \eqref{eq:phi2}.
Again the claim follows by choosing $\delta$ small enough.
\end{proof}
Under the assumptions of Lemma \ref{lem:reg} (b) equation \eqref{eq:SNS} is satisfied strongly in the analytical sense. That is
we have
\begin{align}\label{eq:strong}
\bfu(t)&=\bfu(0)+\int_0^t\Big[\mu\Delta\bfu-(\nabla\bfu)\bfu
-\nabla\pi_{\mathrm{det}}\Big]\ds+
\int_0^t\big[\Phi(\bfu)+\Phi^\pi\big]\,\dd W
\end{align}
$\p$-a.s. for all $t\in[0,T]$, recall equation \eqref{eq:pressure}. In the following we analyze the regularity of the pressure components $\pi_{\mathrm{det}}$ and $\Phi^\pi$.
In the following the subscript ${w^*}$ denotes Bochner-measurability with respect to the weak$^*$-topology.
\begin{corollary}\label{cor:pressure}
\begin{enumerate}
\item Under the assumptions of Lemma \ref{lem:reg} we have
\begin{align*}
\pi_{\mathrm{det}}\in L^{\frac{r}{2}}(\Omega,L^2(0,T;L^{2}(\mt)),\\
\Phi^\pi\in L^r(\Omega;L^\infty_{w^*}(0,T;L_2(\mathfrak U;L^{2}(\mt)))).
\end{align*} 
\item Under the assumptions of Lemma \ref{lem:reg} (b) we have
\begin{align*}
\pi_{\mathrm{det}}\in L^{\frac{r}{2}}(\Omega,L^2(0,T;W^{1,2}(\mt)),\\
\Phi^\pi\in L^r(\Omega;L^\infty_{w^*}(0,T;L_2(\mathfrak U;W^{1,2}(\mt)))).
\end{align*} 
\item Under the assumptions of Lemma \ref{lem:reg} (c) we have 
\begin{align*}
\pi_{\mathrm{det}}\in L^{\frac{r}{2}}(\Omega,L^2(0,T;W^{2,2}(\mt)),\\
\Phi^\pi\in L^r(\Omega;L^\infty_{w^*}(0,T;L_2(\mathfrak U;W^{2,2}(\mt)))).
\end{align*} 
\end{enumerate}
\end{corollary}
\begin{proof}
The key tool is the continuity of $\Delta^{-1}$ (from $W^{k-2,p}(\mt)\rightarrow W^{k,p}(\mt)$ for all $k\in\mathbb N_0$ and $1<p<\infty$). So, the regularity 
transfers from $\bfu\otimes\bfu$ to $\pid$.
We obtain by Ladyshenskaya's inequality
\begin{align*}
\int_0^T\int_{\mt}|\pi_{\mathrm{det}}|^2\dxt&=\int_0^T\int_{\mt}|\Delta^{-1}\Div\Div(\bfu\otimes\bfu)|^2\dxt\leq\,c\,\int_0^T\int_{\mt}|\bfu|^4\dxt\\
&\leq\,c\,\|\bfu\|^2_{L^\infty_tL^2_x}\|\nabla\bfu\|^2_{L^2_tL^{2}_x}
\leq\,c\,\Big(\|\bfu\|^4_{L^\infty_tL^2_x}+\|\nabla\bfu\|^4_{L^2_tL^{2}_x}\Big)
\end{align*}
such that
\begin{align*}
\E\bigg[\int_0^T\int_{\mt}|\pi_{\mathrm{det}}|^2\dxt\bigg]^{\frac{r}{4}}&
\leq\,c\,\E\Big[\|\bfu\|^{2}_{L^\infty_tL^2_x}+\|\nabla\bfu\|^{2}_{L^2_tL^{2}_x}\Big]^{\frac{r}{2}}<\infty
\end{align*}
using Lemma \ref{lem:reg} (a). The estimates for (b) and (c) are very similar. We have to compare $\nabla\pi$ with $|\bfu||\nabla\bfu|$, where we have
\begin{align*}
|\bfu||\nabla\bfu|\leq|\bfu|^2+|\nabla\bfu|^2\in L^{\frac{r}{2}}(\Omega;L^2(Q))
\end{align*}
under the assumptions of Lemma \ref{lem:reg} (b).
Similarly, $\nabla^2\pi$ behaves as $|\bfu||\nabla^2\bfu|+|\nabla\bfu|^2$
which belongs to the same class since Lemma \ref{lem:reg} (c) applies.\\
Now, we investigate the regularity of $\Phi^\pi$. We use continuity $\Delta^{-1}$
to obtain
\begin{align*}
\E\bigg[\sup_{0\leq t\leq T}\|\Phi^\pi\|_{L_2(\mathfrak U;L^2_x)}^2\bigg]^{\frac{r}{2}}&=\E\bigg[\sup_{0\leq t\leq T}\|\nabla\Delta^{-1}\Div\Phi(\bfu)\|_{L_2(\mathfrak U;L^2_x)}^2\bigg]^{\frac{r}{2}}\\
&=\E\bigg[\sup_{0\leq t\leq T}\bigg[\sum_{k\geq1}\|\nabla\Delta^{-1}\Div\bfg_k(\cdot,\bfu)\|_{L^2_x}^2\bigg]^{\frac{r}{2}}\\
&\leq\,c\,\E\bigg[\sup_{0\leq t\leq T}\sum_{k\geq1}\int_0^T\|\bfg_k(\cdot,\bfu)\|_{L^2_x}^2\bigg]^{\frac{r}{2}}\\
&\leq\,c\,\E\bigg[1+\sup_{0\leq t\leq T}\|\bfu\|_{L^2_x}^2\bigg]^{\frac{r}{2}}<\infty
\end{align*}
using \eqref{eq:phi} and Lemma \ref{lem:reg} (a). Similarly, we have
\begin{align*}
\E\bigg[\sup_{0\leq t\leq T}\|\Phi^\pi\|_{L_2(\mathfrak U;W^{1,2}_x)}^2\bigg]^{\frac{r}{2}}
&\leq\,c\,\E\bigg[\sup_{0\leq t\leq T}\sum_{k\geq1}\|\bfg_k(\cdot,\bfu)\|_{W^{1,2}_x}^2\bigg]^{\frac{r}{2}}\\
&\leq\,c\,\E\bigg[1+\sup_{0\leq t\leq T}\|\bfu\|_{W^{1,2}_x}^2\bigg]^{\frac{r}{2}}<\infty
\end{align*}
using \eqref{eq:phi} and Lemma \ref{lem:reg} (b), as well as
\begin{align*}
\E\bigg[\sup_{0\leq t\leq T}\|\Phi^\pi\|_{L_2(\mathfrak U;W^{2,2}_x)}^2\bigg]^{\frac{r}{2}}
&\leq\,c\,\E\bigg[\sup_{0\leq t\leq T}\sum_{k\geq1}\|\bfg_k(\cdot,\bfu)\|_{W^{2,2}_x}^2\dt\bigg]^{\frac{r}{2}}\\
&\leq\,c\,\E\bigg[1+\sup_{0\leq t\leq T}\|\bfu\|_{W^{2,2}_x}^2\bigg]^{\frac{r}{2}}<\infty
\end{align*}
by \eqref{eq:phi2} and the Lemma \ref{lem:reg} (c).
\end{proof}
Finally, we investigate the time-regularity of the velocity field.

\begin{corollary}\label{cor:uholder}
\begin{enumerate}
\item Let the assumptions of Lemma \ref{lem:reg} (b) be satisfied for some $r>2$. Then we have
\begin{align}
\label{eq:holder}
\E\Big[\|\bfu\|_{C^\alpha([0,T];L^{2}_x)}\Big]^{\frac{r}{2}}<\infty
\end{align}
for all $\alpha<\frac{1}{2}$.
\item Let the assumptions of Lemma \ref{lem:reg} (c) be satisfied for some $r>2$. Then we have
\begin{align}
\label{eq:holder}
\E\Big[\|\bfu\|_{C^\alpha([0,T];W^{1,2}_x)}\Big]^{\frac{r}{2}}<\infty
\end{align}
for all $\alpha<\frac{1}{2}$.
\end{enumerate}
\end{corollary}
\begin{proof}
We start with (a) and analyse each term in equation \eqref{eq:strong} separately. Lemma \ref{lem:reg} (b) implies
\begin{align*}
\int_0^\cdot\Delta\bfu\ds\in L^r(\Omega;W^{1,2}(0,T;L^2(\mt))).
\end{align*}
As seen in the proof of Corollary \ref{cor:pressure} $\pid$ and $\bfu\otimes\bfu$ have the same regularity. In particular, Corollary \ref{cor:pressure} (b) yields
\begin{align*}
\int_0^\cdot\big(\Div(\bfu\otimes\bfu)+\nabla\pid\big)\ds\in L^{\frac{r}{2}}(\Omega;W^{1,2}(0,T;L^2(\mt))).
\end{align*}
Finally, we have
\begin{align*}
\int_0^\cdot\Phi(\bfu)\,\dd W\in L^{r}(\Omega;C^{\alpha}([0,T];L^2(\mt))).
\end{align*}
by combing Lemma \ref{lems:Holder} with \eqref{norm1}. The same conclusion holds for $\Phi^\pi$ using Lemma \ref{cor:pressure} (a).
Plugging all together and noting the embedding $W^{1,2}(0,T;X)\hookrightarrow C^{\alpha}([0,T];X)$ for any separably Banach space $X$ the claim follows.\\
The proof of (b) follows along the same lines using the higher regularity from
Lemma \ref{lem:reg} (c), Corollary \ref{cor:pressure} (c) and \eqref{norm2}.
\end{proof}

\subsection{Discretization in space}
We work with a standard finite element set-up for incompressible fluid mechanics, see e.g. \cite{BF} and \cite{GR}.
We denote by $\mathscr{T}_h$ a quasi-uniform subdivision of $\mt$ into triangles of maximal diameter $h>0$.   
For $\mathcal S\subset \setR ^2$ and $\ell\in \setN _0$ we denote by
$\mathscr{P}_\ell(\mathcal S)$ the polynomials on $\mathcal S$ of degree less than or equal
to $\ell$. Moreover, we set $\mathscr{P}_{-1}(\mathcal S):=\set {0}$. Let us
characterize the finite element spaces $V^h(\mt)$ and $P^h(\mt)$ as
\begin{align*}
  V^h(\mt)&:= \set{\bfv_h \in W^{1,2}(\mt)\,:\, \bfv_h|_{\mathcal S}
    \in \mathscr{P}_i(\mathcal S)\,\,\forall \mathcal S\in \mathscr T_h},\\
P^h(\mt)&:=\set{\pi_h \in L^{2}(\mt)\,:\, \pi_h|_{\mathcal S}
    \in \mathscr{P}_j(\mathcal S)\,\,\forall \mathcal S\in \mathscr T_h}.
\end{align*}
We will assume that $i$ and $j$ are both natural to get \eqref{eq:stab'} below (this is different from \cite{CP}, where also $j=0$ is allowed).
In order to guarantee stability of our approximations we relate $V^h(\mt)$ and $P^h(\mt)$ by the inf-sup condition, that is we assume that
\begin{align*}
\sup_{\bfv_h\in V^h(\mt)} \frac{\int_{\mt}\Div\bfv_h\,\pi_h\dx}{\|\nabla\bfv_h\|_{L^2_x}}\geq\,C\,\|\pi_h\|_{L^2_x}\quad\,\forall\pi_h\in P^h(\mt),
\end{align*}
where $C>0$ does not depend on $h$. This gives a relation between $i$ and $j$ (for instance the choice $(i,j)=(1,0)$ is excluded whereas $(i,j)=(2,0)$ is allowed).
Finally, we define the space of discretely solenoidal finite element functions by 
\begin{align*}
  V^h_{\Div}(\mt)&:= \bigg\{\bfv_h\in V^h(\mt):\,\,\int_{\mt}\Div\bfv_h\,\,\pi_h\dx=0\,\forall\pi_h\in P^h(\mt)\bigg\}.
\end{align*}
Let $\Pi_h:L^2(\mt)\rightarrow V_{\Div}^h(\mt)$ be the $L^2(\mt)$-orthogonal projection onto $V_{\Div}^h(\mt)$. The following results concerning the approximability of $\Pi_h$ are well-known (see, for instance \cite{HR}). There is $c>0$ independent of $h$ such that we have
  \begin{align}
    \label{eq:stab}
 \int_{\mt} \Big|\frac{\bfv-\Pi_h \bfv}{h}\Big|^2\dx+ \int_{\mt} \abs{\nabla\bfv-\nabla\Pi_h \bfv}^2\dx &\leq
    \,c\, \int_{\mt} \abs{\nabla
      \bfv}^2\dx
  \end{align}
for all $\bfv\in W^{1,2}_{\Div}(\mt)$ and
  \begin{align}
    \label{eq:stab'}
 \int_{\mt} \Big|\frac{\bfv-\Pi_h \bfv}{h}\Big|^2\dx+ \int_{\mt} \abs{\nabla\bfv-\nabla\Pi_h \bfv}^2\dx &\leq
    \,c\,h^2 \int_{\mt} \abs{\nabla^2
      \bfv}^2\dx
  \end{align}
for all $\bfv\in W^{2,2}_{\Div}(\mt)$. Similarly, if $\Pi_h^\pi:L^2(\mt)\rightarrow P^h(\mt)$ denotes the $L^2(\mt)$-orthogonal projection onto $P^h(\mt)$, we have
\begin{align}
\label{eq:stabpi}
 \int_{\mt} \Big|\frac{p-\Pi_h^\pi p}{h}\Big|^2\dx &\leq
    \,c\, \int_{\mt} \abs{\nabla
      p}^2\dx
\end{align}
for all $p\in W^{1,2}(\mt)$ and
\begin{align}
\label{eq:stabpi'}
 \int_{\mt} \Big|\frac{p-\Pi_h^\pi p}{h}\Big|^2\dx &\leq
    \,c\,h^2 \int_{\mt} \abs{\nabla^2
      p}^2\dx
\end{align}
for all $p\in W^{2,2}(\mt)$. Note that \eqref{eq:stabpi'} requires the assumption $j\geq1$ in the definition of $P^h(\mt)$, whereas \eqref{eq:stabpi} also holds for $j=0$.

\section{Time-discretization}
We consider an equidistant partition of $[0,T]$ with mesh size $\Delta t=T/M$ and set $t_m=m\Delta t$. 
Let $\bfu_0$ be an $\mathfrak F_0$-mesureable random variable with values in $W^{1,2}_{\Div}(\mt)$. We aim at constructing iteratively a sequence of $\mathfrak F_{t_m}$-measurable random variables $\bfu_{m}$ with values in $W^{1,2}_{\Div}(\mt)$ such that
for every $\bfphi\in W^{1,2}_{\Div}(\mt)$ it holds true $\p$-a.s.
\begin{align}\label{tdiscr}
\begin{aligned}
\int_{\mt}&\bfu_{m}\cdot\bfvarphi \dx +\Delta t\bigg(\int_{\mt}(\nabla\bfu_m)\bfu_{m}\cdot\bfphi\dx+\mu\int_{\mt}\nabla\bfu_{m}:\nabla\bfphi\dx\bigg)\\
&\qquad=\int_{\mt}\bfu_{m-1}\cdot\bfvarphi \dx+\int_{\mt}\Phi(\bfu_{m-1})\,\Delta_mW\cdot \bfvarphi\dx,
\end{aligned}
\end{align}
where $\Delta_m W=W(t_m)-W(t_{m-1})$. The existence of a unique $\bfu_m$ (given $\bfu_m$ and $\Delta_m W$) solving \eqref{tdiscr} is straightforward as it is a stationary Navier--Stokes system. The following result follows from Lemma 3.1 in \cite{CP}.
\begin{lemma}\label{lemma:3.1} 
Assume that $\bfu_0\in L^{2^q}(\Omega,W^{1,2}_{div}(\mt))$ for some $1\leq q<\infty$. Suppose that $\Phi$ satisfies \eqref{eq:phi}. Then the iterates $(\bfu_m)_{m=1}^M$ given by \eqref{tdiscr}
are $\mathfrak F_{t_m}$-measurable. Moreover, the following estimates hold uniformly in $M$:
\begin{align}
\label{lem:3.1a}\E\bigg[\max_{1\leq m\leq M}\|\bfu_m\|^{2^q}_{W^{1,2}_x}+\Delta t\sum_{k=1}^M\|\bfu_m\|_{W^{1,2}_x}^{2^q-2}\|\nabla^2\bfu_m\|^2_{L^2_x}\bigg]&\leq\,c(q,T,\bfu_0),\\
\label{lem:3.1b}\E\bigg[\sum_{k=1}^M\|\bfu_m-\bfu_{m-1}\|^2_{W^{1,2}_x}\|\nabla\bfu_m\|^2_{L^2_x}\bigg]&\leq\,c(T,\bfu_0),\\
\label{lem:3.1c}\E\bigg[\bigg(\sum_{k=1}^M\|\bfu_m-\bfu_{m-1}\|^2_{W^{1,2}_x}\bigg)^4+\bigg(\Delta t\sum_{k=1}^M\|\nabla\bfu_m\|^2_{L^2_x}\bigg)^4\bigg]&\leq\,c(T,\bfu_0).
\end{align}
\end{lemma}

Now, for $\bfphi\in W^{1,2}(\mt)$ we can insert $\bfphi-\nabla\Delta^{-1}\Div\bfphi\in W^{1,2}_{\Div}(\mt)$ in \eqref{tdiscr}
and obtain
\begin{align}
\nonumber
\int_{\mt}\bfu_m\cdot\bfvarphi\dx &+\Delta t\bigg(\int_{\mt}(\nabla\bfu_m)\bfu_{m}\cdot\bfphi\dx+\mu\int_{\mt}\nabla\bfu_{m}:\nabla\bfphi\dx\bigg)\\
\label{eq:pressure}&=\int_{\mt}\bfu_{m-1}\cdot\bfvarphi\dx
+\Delta t\int_{\mt}\pi_m^{\mathrm{det}}\,\Div\bfphi\dx
\\
\nonumber&+\int_{\mt}\Phi(\bfu_{m-1})\,\Delta_m W\cdot \bfvarphi\dx+\int_{\mt}\int_0^t\Phi^\pi_{m-1}\,\Delta_m W\cdot \bfvarphi\dx,
\end{align}
where 
\begin{align*}
\pi_m^{\mathrm{det}}&=-\Delta^{-1}\Div\Div\big(\bfu_m\otimes\bfu_m\big),\\
\Phi_{m-1}^\pi&=-\nabla\Delta^{-1}\Div\Phi(\bfu_{m-1}).
\end{align*}
\begin{lemma}\label{lemma:3.2}
Assume that $\bfu_0\in L^{8}(\Omega,W^{1,2}_{div}(\mt))$ and that $\Phi$ satisfies \eqref{eq:phi}.
For all $m\in\{1,...,M\}$ the random variable $\pi_m^{\mathrm{det}}$ is $\mathfrak F_{t_m}$-measureable, has values in $W^{1,2}(\mt)$ and we have
\begin{align*}
\E\bigg[\Delta t \sum_{m=1}^M\big\|\nabla \pi_m^{\mathrm{det}}\big\|_{L^2_x}^2\bigg]\leq \,C
\end{align*}
uniformly in $\Delta t$. 
\end{lemma}
\begin{remark}
A corresponding result is shown in \cite[Lemma 3.2]{CP} for the full pressure provided the noise is divergence-free (in this case $\Phi_m^\pi$ vanishes). This is quite restrictive as it means the $\Phi$ is linear in $\bfu$.
\end{remark}
\begin{proof}
The $\mathfrak F_{t_m}$-measurability of $\pi_m^{\mathrm{det}}$ follows directly from the one of $\bfu_m$ stated in Lemma \ref{lemma:3.1}.
By continuity of the operator $\nabla\Delta^{-1}\Div$ on $L^2(\mt)$ we have
\begin{align*}
\big\|\nabla \pi_m^{\mathrm{det}}\big\|_{L^2_x}^2&\leq\,c\,\big\|\Div(\bfu_m\otimes\bfu_m)\big\|^2_{L^2_x}\leq\,c\,\|\bfu_m\|_{{L^4_x}}^2\|\nabla\bfu_m\|_{{L^4_x}}^2\\
&\leq\,c\,\|\bfu_m\|_{{L^2_x}}\|\nabla\bfu_m\|_{{L^2_x}}^2\|\nabla^2\bfu_m\|_{{L^2_x}}\leq\,c\,\|\bfu_m\|_{{L^2_x}}^2+c\,\|\nabla\bfu_m\|_{{L^2_x}}^4\|\nabla^2\bfu_m\|_{{L^2_x}}^2
\end{align*}
$\p$-a.s. using also Ladyshenskaya's inequality $\|v\|_{L^4_x}^2\leq\,c\,\|v\|_{L^2_x}\|\nabla v\|_{L^2_x}$ which holds in two-dimensions. Now, summing with respect to $m$, applying expectations and using Lemma \ref{lemma:3.1} yields the claim.
\end{proof}
\begin{lemma}\label{lemma:3.3}
For all $m\in\{1,...,M\}$ the random variable $\Phi^\pi_{m}$ is $\mathfrak F_{t_m}$-measureable, has values in $L_2(\mathfrak U;W^{1,2}(\mt))$ and we have
\begin{align*}
\E\bigg[\Delta t \sum_{m=1}^M\big\|\Phi^\pi_m\big\|_{L_2(\mathfrak U;W^{1,2}_x)}^2\bigg]\leq \,C
\end{align*}
uniformly in $\Delta t$. 
\end{lemma}
\begin{proof}
As for Lemma \ref{lemma:3.2} the proof mainly relies on the continuity of $\nabla\Delta^{-1}\Div$ on $L^2(\mt)$. Here, we have by \eqref{norm2}
\begin{align*}
\|\Phi^\pi_m\|^2_{L_2(\mathfrak U;W^{1,2}_x)}&=\sum_{k\geq1}\|\nabla\Delta^{-1}\Div\big(\Phi(\bfu_m)e_k\big)\|^2_{W^{1,2}_x}\\
&\leq\,c\,\sum_{k\geq1}\|\Phi(\bfu_m)e_k\|^2_{W^{1,2}_x}=c\,\|\Phi(\bfu_m)\|^2_{L_2(\mathfrak U;W^{1,2}_x)}\leq\,c\,\big(1+\|\nabla\bfu_m\|^2_{W^{1,2}_x}\big).
\end{align*}
Summing over $m$, applying expectations and using Lemma \ref{lemma:3.1}
finishes the proof.
\end{proof}
Following \cite{CP} we set for $\varepsilon>0$
\begin{align}\label{eq:ODT}
\Omega^\varepsilon_{\Delta t}=\Big\{\omega\in\Omega\Big|\max_{1\leq m\leq M}\|\nabla \bfu_m\|^2_{L^2_x}\leq -\varepsilon\log (\Delta t)\Big\}
\end{align}
such that
\begin{align*}
\p\big(\Omega_{\Delta t}^\varepsilon\big)\geq 1-\frac{\E\big[\max_{1\leq m\leq M}\|\nabla \bfu_m\|_{L^2_x}^2\big]}{-\varepsilon\log (\Delta t)}\geq1+\frac{C}{\varepsilon\log \Delta t}
\end{align*}
using Lemma \ref{lemma:3.1}. We obtain the following result.
\begin{theorem}\label{thm:3.1}
Assume that \eqref{eq:phi} holds and that $\bfu_0\in L^{8}(\Omega,W^{1,2}_{div}(\mt))$ is an $\mathfrak F_0$-measureable random variable. Let $\bfu$ be the unique strong solution
to \eqref{eq:SNS} in the sense of Definition \ref{def:inc2d}. Assume that we have
\begin{align}
\E\Big[\|\bfu\|_{C^\alpha([0,T];L^{4}_x)}^4\Big]<\infty,\quad
\E\Big[\|\bfu\|_{C^\alpha([0,T];W^{1,2}_x)}^2\Big]<\infty,\label{eq:holder'}
\end{align}
for some $\alpha\in(0,\tfrac{1}{2})$; recall Corollary \ref{cor:uholder}. Let $(\bfu_m)_{m=1}^M$ be the solution to \eqref{tdiscr}. Then we have the error estimate
\begin{align}\label{eq:thm:4}
\begin{aligned}
\E\bigg[\mathbf{1}_{\Omega_{\Delta t}^\varepsilon}\bigg(\max_{1\leq m\leq M}\|\bfu(t_m)-\bfu_{m}\|^2_{L^2_x}&+\Delta t\sum_{m=1}^M  \|\nabla\bfu(t_m)-\nabla\bfu_{m}\|^2_{L^2_x}\bigg)\bigg]\leq \,c_\varepsilon\,(\Delta t)^{2\alpha-\varepsilon}
\end{aligned}
\end{align}
for any $\varepsilon>0$.
\end{theorem}
\begin{remark}
\begin{itemize}
\item Estimate \eqref{eq:thm:4} improves the result from \cite[Thm. 3.1]{CP}, where the convergence rate is only $\alpha-\varepsilon$.
\item In the paper \cite{BeMi} the time-discretization of the stochastic Navier--Stokes equations is analysed. The corresponding error does not contain an indicator function such as $\mathbf{1}_{\Omega_{\Delta t}^\varepsilon}$. The convergence
 rate is, however, only of logarithmic order.
 \end{itemize}
\end{remark}
\begin{proof}
Subtracting \eqref{eq:SNS} and \eqref{tdiscr} we obtain the equation for the error $\bfe_m=\bfu(t_m)-\bfu_m$
which reads as
\begin{align}\label{terror}
\begin{aligned}
\int_{\mt}&\bfe_{m}\cdot\bfvarphi \dx +\int_{t_{m-1}}^{t_m}\bigg(\int_{\mt}\big((\nabla\bfu)\bfu-(\nabla\bfu_m)\bfu_{m}\big)\cdot\bfphi\dx+\mu\int_{\mt}\nabla\bfe_{m}:\nabla\bfphi\dx\bigg)\dt\\
&\qquad=\int_{\mt}\bfe_{m-1}\cdot\bfvarphi \dx+\mu\int_{t_{m-1}}^{t_m}\int_{\mt}\big(\nabla \bfu(t_m)-\nabla \bfu(t)\big):\nabla\bfphi\dxt\\
&\qquad+\int_{\mt}\bigg(\int_{t_{m-1}}^{t_m}\Phi(\bfu)\,\dd W-\Phi(\bfu_{m-1})\,\Delta_mW\bigg)\cdot \bfvarphi\dx
\end{aligned}
\end{align}
for all $\bfphi\in W^{1,2}_{\Div}(\mt)$. Choosing $\bfphi=\bfe_m$ implies
\begin{align*}
\frac{1}{2}\int_{\mt}&|\bfe_{m}|^2 \dx+\frac{1}{2}\int_{\mt}|\bfe_{m}-\bfe_{m-1}|^2 \dx +\mu\Delta t\int_{\mt}|\nabla\bfe_{m}|^2\dx\\
&=\frac{1}{2}\int_{\mt}|\bfe_{m-1}|^2 \dx
-\mu\int_{t_{m-1}}^{t_m}\int_{\mt}\big(\nabla \bfu(t_m)-\nabla \bfu(t)\big):\nabla\bfe_{m}\dxt\\&+\int_{t_{m-1}}^{t_m}
\int_{\mt}\big((\nabla\bfu)\bfu-(\nabla\bfu_m)\bfu_{m}\big)\cdot\bfe_m\dx\\&+\int_{\mt}\bigg(\int_{t_{m-1}}^{t_m}\Phi(\bfu)\,\dd W-\Phi(\bfu_{m-1})\,\Delta_mW\bigg)\cdot \bfe_m\dx.
\end{align*}
Iterating this equality and following the arguments in \cite[proof of Thm. 3.1]{CP}
yields
\begin{align}\label{eq:66}
\begin{aligned}
\E\bigg[&\mathbf{1}_{\Omega_{\Delta t}^\varepsilon}\bigg(\max_{1\leq m\leq M}\|\bfu(t_m)-\bfu_{m}\|^2_{L^2_x}+\Delta t\sum_{m=1}^M  \|\nabla\bfu(t_m)-\nabla\bfu_{m}\|^2_{L^2_x}\bigg)\bigg]\\&\leq \,c\,(\Delta t)^{2\alpha-\varepsilon}
+c\,\E\bigg[\sum_{m=1}^M\int_{t_{m-1}}^{t_m}\int_{\mt}\big|\nabla \bfu(t_m)-\nabla \bfu(t)\big|^2\dxt\bigg].
\end{aligned}
\end{align}
Note that only the first bound from \eqref{eq:holder'} has been used for \eqref{eq:66}. The second bound from \eqref{eq:holder'}, not used in \cite{CP},
allows us to estimates the remaining integral by
\begin{align*}
\E\bigg[\sum_{m=1}^M\int_{t_{m-1}}^{t_m}\int_{\mt}\big|\nabla \bfu(t_m)-\nabla \bfu(t)\big|^2\dxt\bigg]\leq \,c\,(\Delta t)^{2\alpha}
\end{align*}
which finishes the proof.
\end{proof}

\section{Finite element based space-time approximation}

Now we consider a fully practical scheme combining the implicite Euler scheme in time (as in the last section) with a finite element approximation in space. 
For a given $h>0$ let $\bfu_{h,0}$ be an $\mathfrak F_0$-mesureable random variable with values in $V^h_{\Div}(\mt)$ (for instance $\Pi_h\bfu_0$). We aim at constructing iteratively a sequence of random variables $\bfu_{h,m}$ with values in $V^h_{\Div}(\mt)$ such that
for every $\bfphi\in V^h_{\Div}(\mt)$ it holds true $\p$-a.s.
\begin{align}\label{txdiscr}
\begin{aligned}
\int_{\mt}&\bfu_{h,m}\cdot\bfvarphi \dx +\Delta t\int_{\mt}\big((\nabla\bfu_{h,m})\bfu_{h,m-1}+(\Div\bfu_{h,m-1})\bfu_{h,m}\big)\cdot\bfphi\dx\\
&+\mu\,\Delta t\int_{\mt}\nabla\bfu_{m}:\nabla\bfphi\dx=\int_{\mt}\bfu_{h,m-1}\cdot\bfvarphi \dx+\int_{\mt}\Phi(\bfu_{h,m-1})\,\Delta_mW\cdot \bfvarphi\dx,
\end{aligned}
\end{align}
where $\Delta_m W=W(t_m)-W(t_{m-1})$. We quote the following result concerning the existence of solutions $\bfu_{h,m}$ to \eqref{txdiscr} from \cite[Lemma 3.1]{BCP}.

\begin{lemma}\label{lemma:3.1BCP}
Let $1\leq q<\infty$. Assume that $\bfu_{h,0}\in L^{2^q}(\Omega,V_{\Div}^h(\mt))$ is an $\mathfrak F_0$-mesureable random variable and that $\E\big[\|\bfu_{h,0}\|^{2^q}_{L^2_x}\big|]\leq \,K$ uniformly in $h$ for some $K>0$. Suppose that $\Phi$ satisfies \eqref{eq:phi}. Then the iterates $(\bfu_{h,m})_{m=1}^M$ given by \eqref{txdiscr}
are $\mathfrak F_{t_m}$-measurable. Moreover, the following estimate holds uniformly in $M$ and $h$:
\begin{align}
\label{lem:4.1}\E\bigg[\max_{1\leq m\leq M}\|\bfu_{h,m}\|^{2^q}_{L^{2}_x}+\Delta t\sum_{k=1}^M\|\bfu_{h,m}\|^{2^{q}-2}_{L^{2}_x}\|\nabla\bfu_{h,m}\|^2_{L^2_x}\bigg]&\leq\,c(q,T,K).
\end{align}
\end{lemma}

\subsection{Error analysis}
\label{sec:error}

In this subsection we establish convergence with rates of the above defined algorithm. We introduce for $\varepsilon>0$ the sample set
\begin{align*}
\Omega^\varepsilon_{h}=\Big\{\omega\in\Omega\Big|\max_{0\leq m\leq M}\Big(\|\nabla \bfu_m\|^4_{L^2_x}+\|\bfu_{h,m}\|_{L^2_x}^2\Big)\leq -\varepsilon\log h\Big\}
\end{align*}
which can be controlled by
\begin{align*}
\p\big(\Omega^\varepsilon_{h}\big)\geq 1-\frac{\E\big[\max_{1\leq m\leq M}\big(\|\nabla \bfu_m\|^4_{L^2_x}+\|\bfu_{h,m}\|_{L^2_x}^2\big)\big]}{-\varepsilon\log h}\geq1+\frac{C}{\varepsilon\log h}
\end{align*}
using Lemma \ref{lemma:3.1} and \ref{lemma:3.1BCP}. We also recall the definition of $\Omega^\varepsilon_{\Delta t}$ in \eqref{eq:ODT}. We are now ready to state our main result.

\begin{theorem}\label{thm:4}
Let $\bfu_0\in L^2(\Omega,W^{1,2}_{\Div}(\mt))\cap L^8(\Omega;L^2_{\Div}(\mt))$ be $\F_0$-measurable and assume that $\Phi$ satisfies \eqref{eq:phi}. Let $\bfu$ be the unique strong solution to \eqref{eq:SNS} in the sense of Definition \ref{def:inc2d}.
Suppose further that
\begin{align*}
\E\Big[\|\bfu\|_{C^\alpha([0,T];L^{4}_x)}^4\Big]<\infty,\quad
\E\Big[\|\bfu\|_{C^\alpha([0,T];W^{1,2}_x)}^2\Big]<\infty,
\end{align*}
for $\alpha\in(0,\tfrac{1}{2})$; recall Corollary \ref{cor:uholder}. Assume that $L\Delta t\leq \,(-\varepsilon\log h)^{-1}$ for some $L>0$.
Then we have
\begin{align}\label{eq:thm:4}
\begin{aligned}
\E\bigg[\mathbf{1}_{\Omega^\varepsilon_{\Delta t}\cap\Omega^\varepsilon_{h}}\bigg(\max_{1\leq m\leq M}\|\bfu(t_m)-\bfu_{h,m}\|_{L^2_x}^2&+\sum_{m=1}^M \Delta t \|\nabla\bfu(t_m)-\nabla\bfu_{h,m}\|_{L^2_x}^2\bigg)\bigg]\\&\leq \,c\,\big(h^2+(\Delta t)^{2\alpha-\varepsilon}\big),
\end{aligned}
\end{align}
where $(\bfu_{h,m})_{m=1}^M$ is the solution to \eqref{txdiscr} with $\bfu_{h,0}=\Pi_h\bfu_0$.
The constant $c$ is \eqref{eq:thm:4} is independent of $M$ and $h$.
\end{theorem}
\begin{remark}
We do not expect that it is possible to avoid an indicator function in general (see \cite{Pr} for the numerical approximation of stochastic PDEs with non-Lipschitz nonlinearities).
But it is not clear if our choice of the sample subset is optimal.
\end{remark}
%
%

The rest of the paper is devoted to the proof of Theorem \ref{thm:4}.
In fact Theorem \ref{thm:4} will follow from combining Theorem \ref{thm:3.1}
with the following theorem which estimates the error between the time-discrete solution $\bfu_m$ to \eqref{tdiscr}
and the solution $\bfu_{h,m}$ to \eqref{txdiscr}.
\begin{theorem}\label{thm:4a}
Let $\bfu_0\in L^2(\Omega;W^{1,2}_{\Div}(\mt))\cap L^8(\Omega;L^2_{\Div}(\mt))$ be $\F_0$-measurable and assume that $\Phi$ satisfies \eqref{eq:phi}. Let $(\bfu_m)_{m=1}^M$ be the solution to \eqref{tdiscr}. Assume that $L\Delta t\leq \,(-\varepsilon\log h)^{-1}$ for some $L>0$.
Then we have
\begin{align}\label{eq:thm:4'}
\begin{aligned}
\E\bigg[\mathbf{1}_{\Omega^\varepsilon_{h}}\bigg(\max_{1\leq m\leq M}\|\bfu_m-\bfu_{h,m}\|_{L^2_x}^2&+\Delta t \sum_{m=1}^M \|\nabla\bfu_m-\nabla\bfu_{h,m}\|_{L^2_x}^2\bigg)\bigg]\\&\leq \,c\,\big(h^2+\Delta t\big),
\end{aligned}
\end{align}
where $(\bfu_{h,m})_{m=1}^M$ is the solution to \eqref{txdiscr} with $\bfu_{h,0}=\Pi_h\bfu_0$.
The constant $c$ is \eqref{eq:thm:4'} is independent of $M$ and $h$.
\end{theorem}
\begin{remark}
In the estimate of \cite[Thm. 4.1]{CP} the error is estimated by $h^{-3\varepsilon}(h^2+(\Delta t)+\tfrac{h^2}{\Delta t})$. By our refined pressure analysis we are able to get rid of the term $\tfrac{h^2}{\Delta t}$ which leads to a restrictive assumption between space and time-discretization. Additionally, we can remove the factor $h^{-3\varepsilon}$ arising from the convective term, see \eqref{eq:1806} below. Note that the error term $\tfrac{h^2}{\Delta t}$ also appears in \cite{CHP}, where  the finite-element based space-time discretization of the stochastic Stokes equations (that is the linearized version of \eqref{eq:SNS} without convective term) is studied. 
\end{remark}

\begin{proof}[Proof of Theorem \ref{thm:4a}]
Define the error $\bfe_{h,m}=\bfu_m-\bfu_{h,m}$. Subtracting \eqref{eq:pressure} and \eqref{txdiscr} we obtain
\begin{align*}
\begin{aligned}
\int_{\mt}&\bfe_{h,m}\cdot\bfvarphi \dx +\Delta t\int_{\mt}\mu\Big(\nabla\bfu_m-\nabla\bfu_{h,m}\Big):\nabla\bfphi\dx\\&=\int_{\mt}\bfe_{h,m-1}\cdot\bfvarphi \dx-\Delta t\int_{\mt}\Big((\nabla\bfu_m)\bfu_m-\big((\nabla\bfu_{h,m})\bfu_{h,m-1}+(\Div\bfu_{h,m-1})\bfu_{h,m}\big)\Big)\cdot\bfphi\dx\\
&+\int_{\mt}\big(\Phi(\bfu_m)-\Phi(\bfu_{h,m-1})\big)\,\Delta_mW\cdot \bfvarphi\dx\\
&-\int_{\mt}\nabla\Delta^{-1}\Div\Phi(\bfu_{m-1})\,\Delta_mW\cdot \bfvarphi\dx+\Delta t\int_{\mt}\pi_m^{\mathrm{det}}\,\Div\bfphi\dx
\end{aligned}
\end{align*}
for every $\bfphi\in V_{\Div}^h(\mt)$.
Setting $\bfphi=\Pi_h\bfe_{h,m}$ and applying the identity $\bfa\cdot(\bfa-\bfb)=\frac{1}{2}\big(|\bfa|^2-|\bfb|^2+|\bfa-\bfb|^2\big)$ (which holds for any $\bfa,\bfb\in\mathbb R^n$) we gain
\begin{align}\label{eq:0207}
\begin{aligned}
\int_{\mt}&\frac{1}{2}\big(|\Pi_h\bfe_{h,m}|^2-|\Pi_h\bfe_{h,m-1}|^2+|\Pi_h\bfe_{h,m}-\Pi_h\bfe_{h,m-1}|^2\big) \dx+\Delta t\int_{\mt}\mu|\nabla\bfe_{h,m}|^2\dx\\
&=\Delta t\int_{\mt}\mu\nabla\bfe_{h,m}:\nabla\big(\bfu_{m}-\Pi_h\bfu_{m}\big)\dx\\
&-\Delta t\int_{\mt}\Big((\nabla\bfu_m)\bfu_m-\big((\nabla\bfu_{h,m})\bfu_{h,m-1}+(\Div\bfu_{h,m-1})\bfu_{h,m}\big)\Big)\cdot\Pi_h\bfe_{h,m}\dx\\
&+\Delta t\int_{\mt}\pi_m^{\mathrm{det}}\,\Div\Pi_h\bfe_{h,m}\dx\\
&+\int_{\mt}\big(\Phi(\bfu_m)-\Phi(\bfu_{h,m-1})\big)\,\Delta_mW\cdot \Pi_h\bfe_{h,m}\dx\\
&-\int_{\mt}\nabla\Delta^{-1}\Div\Phi(\bfu_{m-1})\,\Delta_mW\cdot \Pi_h\bfe_{h,m}\dx\\
&=I_1(m)+\dots +I_5(m).
\end{aligned}
\end{align}
Young's inequality yields
\begin{align*}
I_1(m)&\leq \,\kappa\,\Delta t\int_{\mt}|\nabla\bfe_{h,m}|^2\dx+c_\kappa \Delta t\int_{\mt}|\nabla \bfu_{m}-\nabla\Pi_h\nabla\bfu_{m}|^2\dx\\
&\leq \,\kappa\,\Delta t\int_{\mt}|\nabla\bfe_{h,m}|^2\dx+c_\kappa\,h^2 \Delta t\int_{\mt}|\nabla^2 \bfu_m|^2\dx.
\end{align*} 
for every $\kappa>0$ using also \eqref{eq:stab'}.
The convective term $I_2(m)$ can be decomposed as
\begin{align*}
I_2(m)&=I_2^1(m)+I_2^2(m)+I_2^3(m),\\
I_2^1(m)&=-\Delta t\int_{\mt}(\nabla\bfe_{h,m})\bfu_{m-1}\cdot\big(\bfu_m-\Pi_h\bfu_{m}\big)\dx,\\
I_2^2(m)&=\Delta t\int_{\mt}(\nabla\bfe_{h,m})\bfe_{h,m-1}\cdot\big(\bfu_m-\Pi_h\bfu_{m}\big)\dx\\
&+\Delta t\int_{\mt}(\Div\bfe_{h,m})\bfe_{h,m}\cdot\big(\bfu_m-\Pi_h\bfu_{m}\big)\dx,\\
I_2^3(m)&=-\Delta t\int_{\mt}(\nabla\Pi_h\bfe_{h,m})\bfe_{h,m-1}\cdot\bfu_m\dx\\
&-\Delta t\int_{\mt}(\Div\bfe_{h,m-1})\pi_h\bfe_{h,m}\cdot\bfu_m\dx,
\end{align*}
As in \cite[pages 2489--2491]{CP}
we introduce the sample set
\begin{align*}
\Omega_{h,m}^\varepsilon=\Big\{\omega\in\Omega\Big|\max_{0\leq n\leq m}\Big(\|\nabla \bfu_n\|^4_{L^2_x}+\|\bfu_{h,n}\|_{L^2_x}^2\Big)\leq -\varepsilon\log h\Big\}
\end{align*}
and obtain
\begin{align*}
\mathbf{1}_{\Omega^\varepsilon_{h,m-1}}I_2^2(m)
&\leq\,\kappa\mathbf{1}_{\Omega^\varepsilon_{h,m-1}}\Big(\|\nabla\bfe_{h,m-1}\|^2_{L^2_x}+\|\nabla\bfe_{h,m}\|^2_{L^2_x}\Big)\\
&+c_\kappa\,\Delta t \,h^{4}\Big(\max_{1\leq n\leq m}\|\bfe_{h,n}\|_{L^2_x}^2\Big)\|\nabla\bfu_m\|^2_{L^2_x}\|\nabla^2\bfu_{m}\|_{L^2_x}^2,\\
\mathbf{1}_{\Omega^\varepsilon_{h,m-1}}I_2^3(m)
&\leq\,\kappa\mathbf{1}_{\Omega^\varepsilon_{h,m-1}}\Big(\|\nabla\bfe_{h,m-1}\|^2_{L^2_x}+\|\nabla\bfe_{h,m}\|^2_{L^2_x}\Big)\\
&+c_\kappa\,\Delta t\log(h^{-\varepsilon})\Big(\max_{1\leq n\leq m}\mathbf{1}_{\Omega^\varepsilon_{h,n-1}}\|\bfe_{h,n}\|_{L^2_x}\Big)\\
&+c_\kappa\,\Delta t\|\nabla(\bfu_m-\bfu_{m-1})\|^2_{L^2_x}\Big(\max_{1\leq n\leq m}\|\nabla\bfu_{n}\|_{L^2_x}^2\Big)\Big(\max_{0\leq n\leq m}\|\bfe_{h,n}\|_{L^2_x}^2\Big),
\end{align*}
where $\kappa>0$ is arbitrary. For $I_2^1(m)$, however, we get a slightly better estimate than in \cite{CP} since our definition of $I_2^1(m)$ makes use of $\Div\bfu^m=0$. We get
\begin{align}\label{eq:1806}
\mathbf{1}_{\Omega^\varepsilon_{h,m-1}}I_2^1(m)&\leq\,\kappa\mathbf{1}_{\Omega^\varepsilon_{h,m-1}}\|\nabla\bfe_{h,m}\|^2_{L^2_x}+c_\kappa h^3\log(h^{-\varepsilon})\|\nabla^2\bfu_{m}\|_{L^2_x}^2
\end{align}
due to \eqref{eq:stab}.\\
The crucial point in this proof (making the essential difference to \cite{CP}) are the estimates for the pressure terms
$I_3(m)$ and $I_5(m)$. We will estimate $I_3(m)$ first whereas $I_5(m)$ will only be bounded after iterating \eqref{eq:0207}, see the estimates for $\mathscr M_{m,2}$ below. By \eqref{eq:stab'} we have
\begin{align*}
I_3(m)&=\Delta t\int_{\mt}\big(\pi_m^{\mathrm{det}}-\Pi_h^\pi\pi_m^{\mathrm{det}}\big)\Div\Pi_h\bfe_{h,m}\dx\\
&\leq \,c_\kappa\Delta t \,\int_{\mt}|\pi_m^{\mathrm{det}}-\Pi_h^\pi\pi_m^{\mathrm{det}}|^2\dx+\,\kappa\Delta t\,\int_{\mt}|\nabla\Pi_h\bfe_{h,m}|^2\dx\\
&\leq \, c_\kappa\Delta t h^2\,\int_{\mt}|\nabla\pi_m^{\mathrm{det}}|^2\dx+\,\kappa\Delta t\,\int_{\mt}|\nabla\bfe_{h,m}|^2\dx,
\end{align*}
where $\kappa>0$ is arbitrary.
Plugging all together and choosing $\kappa$ small enough (to absorb the corresponding terms to the left-hand side) we have shown
\begin{align*}
&\mathbf{1}_{\Omega^\varepsilon_{h,m-1}}\bigg(\int_{\mt}|\Pi_h\bfe_{h,m}|^2 \dx +\int_{\mt}|\Pi_h\bfe_{h,m}-\Pi_h\bfe_{h,m-1}|^2 \dx+\Delta t\int_{\mt}|\nabla\bfe_{h,m}|^2\dx\bigg)\\
&\leq \,\mathbf{1}_{\Omega^\varepsilon_{h,m-1}}\int_{\mt}|\bfe_{h,m-1}|^2 \dx
+c\,\Delta t\log(h^{-\varepsilon})\Big(\max_{1\leq n\leq m}\mathbf{1}_{\Omega^\varepsilon_{h,n-1}}\|\bfe_{h,n}\|_{L^2_x}\Big)\\&+\,c\Delta t\,h^2\bigg(\int_{\mt}|\nabla\pi_m^{\mathrm{det}}|^2\dx+\int_{\mt}(1+|\nabla\bfu_m|^2)\dx\int_{\mt}|\nabla^2\bfu_{m}|^2\dx\bigg)\\
&+c\,\Delta t \,h^{4}\Big(\max_{1\leq n\leq m}\|\bfe_{h,n}\|_{L^2_x}^2\Big)\int_{\mt}|\nabla\bfu_m|^2\dx\int_{\mt}|\nabla^2\bfu_{m}|^2\dx\\
&+c\,\Delta t\int_{\mt}|\nabla(\bfu_m-\bfu_{m-1})|^2\dx\bigg(\max_{1\leq n\leq m}\int_{\mt}|\nabla\bfu_{n}|^2\dx\bigg)\bigg(\max_{1\leq n\leq m}\int_{\mt}|\bfe_{h,n}|^2\dx\bigg)\\
&+c\,\mathbf{1}_{\Omega^\varepsilon_{h,m-1}}\int_{\mt}\bigg(\int_{t_{m-1}}^{t_m}\Phi(\bfu_{m-1})-\Phi(\bfu_{h,m-1})\big)\,\dd W\bigg)\cdot \Pi_h\bfe_{h,m}\dx\\
&-c\,\mathbf{1}_{\Omega^\varepsilon_{h,m-1}}\int_{\mt}\bigg(\int_{t_{m-1}}^{t_m}\nabla\Delta^{-1}\Div\Phi(\bfu_{m-1})\,\dd W\bigg)\cdot \Pi_h\bfe_{h,m}\dx.
\end{align*}
Iterating this inequality yields
\begin{align*}
&\mathbf{1}_{\Omega^\varepsilon_{h,m-1}}\int_{\mt}|\Pi_h\bfe_{h,m}|^2 \dx +\sum_{n=1}^m\mathbf{1}_{\Omega^\varepsilon_{h,n-1}}\bigg(\int_{\mt}|\Pi_h\bfe_{h,n}-\Pi_h\bfe_{h,n-1}|^2 \dx+\Delta t\int_{\mt}|\nabla\bfe_{h,n}|^2\dx\bigg)\\
&\leq \,\int_{\mt}|\bfe_{h,0}|^2 \dx+c\,\Delta t\log(h^{-\varepsilon})\sum_{n=1}^m\Big(\max_{1\leq \ell\leq n}\mathbf{1}_{\Omega^\varepsilon_{h,\ell-1}}\|\bfe_{h,\ell}\|^2_{L^2_x}\Big)\\
&+\,c\,h^2\Delta t\sum_{n=1}^m\bigg(\int_{\mt}|\nabla\pi_n^{\mathrm{det}}|^2\dx+\int_{\mt}(1+|\nabla\bfu_n|^2)\dx\int_{\mt}|\nabla^2\bfu_{n}|^2\dx\bigg)\\
&+c\,\Delta t \,h^{4}\sum_{n=1}^m\Big(\max_{1\leq \ell\leq n}\|\bfe_{h,\ell}\|_{L^2_x}^2\Big)\int_{\mt}|\nabla\bfu_n|^2\dx\int_{\mt}|\nabla^2\bfu_{n}|^2\dx\\
&+c\,\Delta t\sum_{n=1}^m\int_{\mt}|\nabla(\bfu_n-\bfu_{n-1})|^2\dx\bigg(\max_{1\leq \ell\leq n}\int_{\mt}|\nabla\bfu_{\ell}|^2\dx\bigg)\bigg(\max_{1\leq \ell\leq n}\int_{\mt}|\bfe_{h,\ell}|^2\dx\bigg)\\
&+c\,\sum_{n=1}^m\mathbf{1}_{\Omega^\varepsilon_{h,n-1}}\int_{\mt}\bigg(\int_{t_{n-1}}^{t_n}\Phi(\bfu_{n-1})-\Phi(\bfu_{h,n-1})\big)\,\dd W\bigg)\cdot \Pi_h\bfe_{h,n}\dx\\
&-c\,\sum_{n=1}^m\mathbf{1}_{\Omega^\varepsilon_{h,n-1}}\int_{\mt}\bigg(\int_{t_{n-1}}^{t_n}\nabla\Delta^{-1}\Div\Phi(\bfu_{n-1})\,\dd W\bigg)\cdot \Pi_h\bfe_{h,n}\dx.
&\quad
\end{align*}
Now, we explain how to bound line by line in expectation. The error in the initial datum can be bounded by $h^2$ by \eqref{eq:stab} and the assumption $\bfu_0\in L^2(\Omega;W^{1,2}_{\Div}(\mt))$. The second term on the right-hand side can be handled by the discrete Gronwall lemma using the assumption $L\Delta t\leq \,(-\varepsilon\log h)^{-1}$. The expectation of the second line is bounded by $h^2$ using Lemmas \ref{lemma:3.1} and \ref{lemma:3.2} (the estimate for $\nabla\pi_m^{\mathrm{det}}$ is the first main ingredient in our proof). This estimate is better than the corresponding one in \cite{CP} as only the deterministic part of the pressure appears here.
The expectation of the third line is bounded by
\begin{align*}
c\,h^{4}\bigg(\E\Big[\max_{1\leq \ell\leq M}\|\bfe_{h,\ell}\|_{L^2_x}^2\Big]^2\bigg)^{\frac{1}{2}}\bigg(\E\bigg[\Delta t \sum_{n=1}^M\int_{\mt}|\nabla\bfu_n|^2\dx\int_{\mt}|\nabla^2\bfu_{n}|^2\dx\bigg]^2\bigg)^{\frac{1}{2}}\leq\,ch^4
\end{align*}
as a consequence of Lemmas \ref{lemma:3.1} and \ref{lemma:3.1BCP}.
The fourth line is controlled by $\Delta t$ due to the estimate
\begin{align*}
\E&\bigg[\sum_{n=1}^m\int_{\mt}|\nabla(\bfu_n-\bfu_{n-1})|^2\dx\bigg(\max_{1\leq \ell\leq n}\int_{\mt}|\nabla\bfu_{\ell}|^2\dx\bigg)\bigg(\max_{1\leq \ell\leq n}\int_{\mt}|\bfe_{h,\ell}|^2\dx\bigg)\bigg]\\
&\leq\bigg(\E\bigg[\sum_{n=1}^M\int_{\mt}|\nabla(\bfu_n-\bfu_{n-1})|^2\dx\bigg]^2\bigg)^{\frac{1}{2}}\bigg(\E\bigg[\max_{1\leq \ell\leq M}\int_{\mt}|\nabla\bfu_{\ell}|^2\dx\bigg]^4\bigg)^{\frac{1}{4}}\times\\
&\qquad\times\bigg(\E\bigg[\max_{1\leq \ell\leq M}\int_{\mt}|\bfe_{h,\ell}|^2\dx\bigg]^4\bigg)^{\frac{1}{4}}
\end{align*}
and the uniform bounds from Lemmas \ref{lemma:3.1} and \ref{lemma:3.1BCP} (with $q=3$).
It remains to estimate the two stochastic terms
\begin{align*}
\mathscr M_{m,1}&=
\sum_{n=1}^m\mathbf{1}_{\Omega^\varepsilon_{h,n-1}}\int_{\mt}\int_{t_{n-1}}^{t_n}\big(\Phi(\bfu_{n-1})-\Phi(\bfu_{h,n-1})\big)\,\dd W\cdot \Pi_h\bfe_{h,n}\dx,\\
\mathscr M_{m,2}&=
\sum_{n=1}^m\mathbf{1}_{\Omega^\varepsilon_{h,n-1}}\int_{\mt}\int_{t_{n-1}}^{t_n}
\big(\mathrm{Id}-\Pi_h^\pi\big)\Delta^{-1}\Div\Phi(\bfu_{n-1})\,\dd W\,\Div \Pi_h\bfe_{h,n}\dx.
\end{align*}
Note that we used integration by parts and $\Pi_h\bfe_{h,n}\in V^h_{\Div}(\mt)$ together with the definition of $\Pi_h^\pi$ to rewrite $\mathscr M_{m,2}$ into the form above. Finally, we write
\begin{align*}
\mathscr M_{m,1}&=
\sum_{n=1}^m\mathbf{1}_{\Omega^\varepsilon_{h,n-1}}\int_{\mt}\int_{t_{n-1}}^{t_n}\big(\Phi(\bfu_{n-1})-\Phi(\bfu_{h,n-1})\big)\,\dd W\cdot \Pi_h\bfe_{h,n-1}\dx\\
&+\sum_{n=1}^m\mathbf{1}_{\Omega^\varepsilon_{h,n-1}}\int_{\mt}\int_{t_{n-1}}^{t_n}\big(\Phi(\bfu_{n-1})-\Phi(\bfu_{h,n-1})\big)\,\dd W\cdot (\Pi_h\bfe_{h,n}-\Pi_h\bfe_{h,n-1})\dx\\
&=:\mathscr M_{m,1}^1+\mathscr M_{m,1}^2
\end{align*}
as well as
\begin{align*}
&\mathscr M_{m,2}=
\sum_{n=1}^m\mathbf{1}_{\Omega^\varepsilon_{h,n-1}}\int_{\mt}\int_{t_{n-1}}^{t_n}
\big(\mathrm{Id}-\Pi_h^\pi\big)\Delta^{-1}\Div\Phi(\bfu_{n-1})\,\dd W\,\Div \Pi_h\bfe_{h,n-1}\dx\\
&+\sum_{n=1}^m\mathbf{1}_{\Omega^\varepsilon_{h,n-1}}\int_{\mt}\int_{t_{n-1}}^{t_n}
\big(\mathrm{Id}-\Pi_h^\pi\big)\Delta^{-1}\Div\Phi(\bfu_{n-1})\,\dd W\,\Div (\Pi_h\bfe_{h,n}-\Pi_h\bfe_{h,n-1})\dx\\
&=:\mathscr M_{m,2}^1+\mathscr M_{m,2}^2.
\end{align*}
These representations have the advantage that $\mathscr M_{m,1}^1$ and $\mathscr M_{m,2}^2$ are martingales (note the index $n-1$ in the indicator functions).
Consequently, we can apply the Burkholder-Davis-Gundy inequality to estimate them. As far as $\mathscr M_{m,1}$ is concerned we have
\begin{align*}
&\E\bigg[\max_{1\leq m\leq M}\big|\mathscr M_{m,1}^1\big|\bigg]\\&\leq\,c\,\E\bigg[\sum_{n=1}^M\mathbf{1}_{\Omega^\varepsilon_{h,n-1}}\int_{t_{n-1}}^{t_n}\|\Phi(\bfu_{n-1})-\Phi(\bfu_{h,n-1})\|^2_{L_2(\mathfrak U,L^2_x)}\|\Pi_h\bfe_{h,n-1}\|^2_{L^2_x}\dt\bigg]^{\frac{1}{2}}\\
&\leq\,c\,\E\bigg[\max_{0\leq n\leq M}\mathbf{1}_{\Omega^\varepsilon_{h,n}}\|\Pi_h\bfe_{h,n}\|_{L^2_x}\bigg(\sum_{n=1}^M\mathbf{1}_{\Omega^\varepsilon_{h,n-1}}\int_{t_{n-1}}^{t_n}\|\Phi(\bfu_{n-1})-\Phi(\bfu_{h,n-1})\|^2_{L_2(\mathfrak U,L^2_x)}\dt\bigg)^{\frac{1}{2}}\bigg]\\
&\leq\,\kappa\,\E\bigg[\max_{0\leq n\leq M}\mathbf{1}_{\Omega^\varepsilon_{h,n}}\|\Pi_h\bfe_{h,n}\|^2_{L^2_x}\bigg]+\,c_\kappa\,\E\bigg[\sum_{n=1}^M\mathbf{1}_{\Omega^\varepsilon_{h,n-1}}\int_{t_{n-1}}^{t_n}\|\bfu_{n-1}-\bfu_{h,n-1}\|_{L^2_x}^2\dt\bigg]\\
&\leq\,\kappa\,\E\bigg[\max_{0\leq n\leq M}\mathbf{1}_{\Omega^\varepsilon_{h,n}}\|\Pi_h\bfe_{h,n}\|^2_{L^2_x}\bigg]+\,c_\kappa\,\E\bigg[\Delta t\sum_{n=1}^M\mathbf{1}_{\Omega^\varepsilon_{h,n-1}}\|\Pi_h\bfe_{h,n-1}\|_{L^2_x}^2\bigg]\\
&+c_\kappa\,\E\bigg[\Delta t\sum_{n=1}^M\mathbf{1}_{\Omega^\varepsilon_{h,n-1}}\|\bfu_{n-1}-\Pi_h\bfu_{n-1}\|_{L^2_x}^2\bigg]
\end{align*}
Here, we also used \eqref{eq:phi}, \eqref{eq:stab} as well as Young's inequality for $\kappa>0$ arbitrary. Now, the first term can be absorbed for $\kappa$ small enough.
The second term can be handled by the discrete Gronwall lemma
(note that $\Omega^\varepsilon_{h,n}\subset \Omega^\varepsilon_{h,n-1}$ such that $\mathbf{1}_{\Omega^\varepsilon_{h,n}}\leq \mathbf{1}_{\Omega^\varepsilon_{h,n-1}}$ a.s.)
Using \eqref{eq:stab} the last term can be estimated by
\begin{align*}
c_\kappa h^2\,\E\bigg[\Delta t\sum_{n=1}^M\mathbf{1}_{\Omega^\varepsilon_{h,n-1}}\|\nabla\bfu_{n-1}\|_{L^2_x}^2\bigg]
\end{align*}
which is bounded by $ch^2$ using Lemma \ref{lemma:3.1} (a) (recall that $\bfu_0\in L^2(\Omega;W^{1,2}(\mt))$).
For the term $\mathscr M_{m,1}^2$ we obtain by Cauchy-Schwartz inequality, Young's inequality, It\^{o}-isometry and \eqref{eq:phi}
\begin{align*}
\E\bigg[\max_{1\leq m\leq M}|\mathscr M_{m,1}^2|\bigg]&\leq \,\kappa\,\E\bigg[ \sum_{n=1}^M\mathbf{1}_{\Omega^\varepsilon_{h,n-1}} \big\|\Pi_h\bfe_{h,n}-\Pi_h\bfe_{h,n-1}\big\|_{L^2_x}^2\bigg]\\ &+c_\kappa\,\E\bigg[\sum_{n=1}^M\mathbf{1}_{\Omega^\varepsilon_{h,n-1}}\bigg\| \int_{t_{n-1}}^{t_n}\big(\Phi(\bfu_{n-1})-\Phi(\bfu_{h,n-1})\big)\,\dd W  \bigg\|_{L^2_x}^2\bigg]\\
&\leq \,\kappa\,\E\bigg[ \sum_{n=1}^M \mathbf{1}_{\Omega^\varepsilon_{h,n-1}}\|\Pi_h\bfe_{h,n}-\Pi_h\bfe_{h,n-1}\|_{L^2_x}^2 \bigg]\\& + c_\kappa\,\E\bigg[\sum_{n=1}^M\mathbf{1}_{\Omega^\varepsilon_{h,n-1}}\int_{t_{n-1}}^{t_n}\|\bfu_{n-1}-\bfu_{h,n-1}\|_{L^2_x}^2\dt\bigg].
\end{align*}
The first term can be absorbed for $\kappa$ small enough whereas the second one can be estimated as for $\mathscr M_{m,1}^1$.\\
Now, we come to the second main ingredient which is the estimate for $\mathscr M_{m,2}$.
We obtain using \eqref{eq:stabpi'}, \eqref{eq:phi} and continuity of $\nabla^2\Delta^{-1}$
\begin{align*}
&\E\bigg[\max_{1\leq m\leq M}\big|\mathscr M_{m,2}^1\big|\bigg]\\&\leq\,c\,\E\bigg[\sum_{n=1}^M\mathbf{1}_{\Omega_{h,n-1}}\int_{t_{n-1}}^{t_n}\|\big(\mathrm{Id}-\Pi_h^\pi\big)\Delta^{-1}\Div\Phi(\bfu_{n-1})\|^2_{L_2(\mathfrak U,L^2_x)}\|\nabla\Pi_h\bfe_{h,n-1}\|^2_{L^2_x}\dt\bigg]^{\frac{1}{2}}\\
&\leq\,c\,h^2\,\E\bigg[\max_{1\leq n\leq M}\|\nabla^2\Delta^{-1}\Div\Phi(\bfu_{n-1})\|_{L_2(\mathfrak U,L^2_x)}\bigg(\sum_{n=1}^M\mathbf{1}_{\Omega_{h,n-1}}\int_{t_{n-1}}^{t_n}\|\nabla\Pi_h\bfe_{h,n}\|^2_{L^2_x}\dt\bigg)^{\frac{1}{2}}\bigg]\\
&\leq\,c_\kappa h^4\,\E\bigg[\max_{1\leq n\leq M}\|\nabla\bfu_{n-1}\|^2_{L^2_x}\bigg]+\,\kappa\,\E\bigg[\Delta t\sum_{n=1}^M\mathbf{1}_{\Omega_{h,n-1}}\|\nabla\Pi_h\bfe_{h,n}\|_{L^{2}_x}^2\bigg].
\end{align*}
The first term is bounded by
$\E\big[\max_{m}\big|\mathscr M_{m,2}^1\big|\big]\leq\,ch^4$ using Lemmas \ref{lemma:3.1} (recall that $\bfu_0\in L^2(\Omega;W^{1,2}(\mt))$). The second term
can be absorbed if $\kappa\ll1$.
As a consequence of Young's inequality, inverse estimates on $V^h(\mt)$, It\^o-ismotry and \eqref{eq:stabpi'} we infer
\begin{align*}
\E\bigg[\max_{1\leq m\leq M}|\mathscr M_{m,2}^2|\bigg]&\leq \,\kappa h^2\,\E\bigg[ \sum_{n=1}^M \mathbf{1}_{\Omega_{h,n-1}}\big\|\nabla\big(\Pi_h\bfe_{h,n}-\Pi_h\bfe_{h,n-1}\big)\big\|_{L^{2}_x}^2\bigg]\\ &+c_\kappa h^{-2}\,\E\bigg[\sum_{n=1}^M\mathbf{1}_{\Omega_{h,n-1}}\bigg\| \int_{t_{n-1}}^{t_n}\big(\mathrm{Id}-\Pi_h^\pi\big)\Delta^{-1}\Div\Phi(\bfu_{n-1})\,\dd W  \bigg\|_{L^2_x}^2\bigg]\\
&\leq \,c\,\kappa \,\E\bigg[ \sum_{n=1}^M\mathbf{1}_{\Omega_{h,n-1}} \big\|\Pi_h\bfe_{h,n}-\Pi_h\bfe_{h,n-1}\big\|_{L^{2}_x}^2\bigg]\\ &+c_\kappa h^{-2}\,\E\bigg[\sum_{n=1}^M\mathbf{1}_{\Omega_{h,n-1}}\int_{t_{n-1}}^{t_n}\big\| \big(\mathrm{Id}-\Pi_h^\pi\big)\Delta^{-1}\Div\Phi(\bfu_{n-1})\big\|_{L_2(\mathfrak U;L^2_x)}^2\dt\bigg]\\
&\leq  \,c\,\kappa \,\E\bigg[ \sum_{n=1}^M\mathbf{1}_{\Omega_{h,n-1}} \big\|\Pi_h\bfe_{h,n}-\Pi_h\bfe_{h,n-1}\big\|_{L^{2}_x}^2\bigg]\\ &+c_\kappa h^{2}\,\E\bigg[\sum_{n=1}^M\int_{t_{n-1}}^{t_n}\big\| \nabla^2\Delta^{-1}\Div\Phi(\bfu_{n-1})\big\|_{L_2(\mathfrak U;L^2_x)}^2\dt\bigg].
\end{align*}
The first term can be absorbed for $\kappa$ small enough. 
 Arguing as for $\mathscr M_{m,2}^1$
the last term is bounded by $ch^2\,\E\big[\max_{n}\|\nabla\bfu_{n-1}\|^2_{L^2_x}\big]$.
Plugging all together and noting that $\Omega^\varepsilon_{h}\subset \bigcup_{n=1}^M \Omega^\varepsilon_{h,n}$ shows
\begin{align*}
\E\bigg[\mathbf{1}_{\Omega^\varepsilon_{h}}\bigg(\max_{1\leq m\leq M}\|\Pi_h\bfe_{h,m}\|_{L^2_x}^2&+\Delta t \sum_{m=1}^M \|\nabla\bfe_{h,m}\|_{L^2_x}^2\bigg)\bigg]\leq \,c\,\big(h^2+\Delta t\big).
\end{align*}
Recalling that $\bfe_{h,m}=\bfu_m-\Pi_h\bfu_m+\Pi_h\bfe_{h,m}$ and using \eqref{eq:stab} as well as Lemma \ref{lemma:3.1} (a) gives the claim.
\end{proof}


\begin{thebibliography}{[M]}
\bibitem{BeTe} A. Bensoussan and R. Temam.
\newblock \'{E}quations stochastiques du type Navier--Stokes.
\newblock J. Funct. Anal. {\bf 13}, 195--222, 1973.
\bibitem{Bi} B. Birnir: The Kolmogorov--Obukhov Statistical Theory
of Turbulence. J. Nonlinear Sci. 23, 657--688. (2013)


\bibitem{BeBrMi} H. Bessaih, Z.~Brze\'zniak, A. Millet (2014): \emph{Splitting up method for the 2D stochastic Navier-Stokes
equations.} Stochastic PDE: Anal. Comput. 2, 433--470.
\bibitem{BeMi} H. Bessaih, A. Millet (2018): \emph{Strong $L^2$ convergence of time numerical schemes for the stochastic two-dimensional Navier--Stokes equations} IMA J. Num. Anal. doi:10.1093/imanum/dry058
\bibitem{Br} D. Breit: Existence theory for stochastic power law fluids. J. Math. Fluid Mech. 17, 295--326. (2015)
\bibitem{Br2} D. Breit: \emph{Existence theory for generalized Newtonian fluids.}
Mathematics in Science and Engineering. Elsevier/Academic Press, London. (2017)
\bibitem{BF} F. Brezzi and M. Fortin, Mixed and Hybrid Finite Element Methods, Springer Ser. Comput.
Math. 15, Springer-Verlag, New York, 1991.
\bibitem{BCP}
Z.~Brze\'zniak, E.~Carelli, J. A.~Prohl (2013): \emph{Finite-element-based discretizations of the incompressible Navier--Stokes
equations with multiplicative random forcing}, IMA J. Num. Anal. 33, 771--824.
\bibitem{CP}
E.~Carelli, A.~Prohl (2012): \emph{Rates of convergence for discretizations of the stochastic incompressible Navier-Stokes equations.} SIAM J. Numer. Anal. 50(5), 2467--2496.
\bibitem{CHP} E. Carelli, E. Hausenblas, A. Prohl (2012): \emph{Time-splitting methods to solve the stochastic incompressible Stokes equation.} SIAM J. Numer. Anal. 50, 2917--2939. 
\bibitem{Ca} M. Capi\'nski, A note on uniqueness of stochastic Navier-Stokes equations, Univ. Iagell. Acta Math. 30 (1993), 219--228.
\bibitem{CC} M.Capi\'nski, N.\,J. Cutland, Stochastic Navier-Stokes equations, Acta Appl. Math. 25 (1991), 59--85.
   \bibitem{PrZa} G. Da Prato, J. Zabczyk (1992): Stochastic Equations in Infinite Dimensions, Encyclopedia
Math. Appl., vol. 44, Cambridge University Press, Cambridge. 
\bibitem{FlaGat} F.~Flandoli and D.~Gatarek
\newblock Martingale and stationary solutions for stochastic Navier-Stokes equations.
\newblock {\em Probab. Theory Relat. Fields}, {\bf 102}, 367--391, 1995.
\bibitem{GR} Girault,V. \& Raviart, P. A. (1981) Finite Element Method for Navier--Stokes Equations: Theory and Algorithms. Berlin/Heidelberg/New York: Springer.
\bibitem{GL} V. Girault and J.L. Lions, Two-grid finite-element schemes for the steady Navier-Stokes
problem in polyhedra, Port. Math. (N.S.), 58 (2001), pp. 25--57.
\bibitem{GS} V. Girault and L.R. Scott, A quasi-local interpolation operator preserving the discrete divergence,
Calcolo, 40 (2003), pp. 1--19.
\bibitem{HR} J.G. Heywood and R. Rannacher, Finite element approximation of the nonstationary
Navier--Stokes problem. I. Regularity of solutions and second-order error estimates for
spatial discretization. SIAM J. Numer. Anal. {\bf 19}, 275--311, 1982.
	\bibitem{Ho1} M. Hofmanov\'{a} (2013): Degenerate Parabolic Stochastic Partial Differential Equations. Stoch. Pr. Ap. 123 (12), 4294--4336.
\bibitem{KukShi}
S.~Kuksin and A.~Shirikyan.
\newblock {\em Mathematics of two-dimensional turbulence}, volume 194 of {\em
  Cambridge Tracts in Mathematics}.
\newblock Cambridge University Press, Cambridge, 2012.
\bibitem{MikRoz} R. Mikulevicius and B. L. Rozovskii. Stochastic Navier-Stokes equations for turbulent flows. SIAM J. Math. Anal., 35(5):1250--1310, 2004.

  \bibitem{Pa} E. Pardoux (1975): Equations aux d\'{e}riv\'{e}es Partielles stochastiques non lin\'{e}aires monotones. Etude de solutions fortes de
type It\^{o}, Ph.D. thesis, Universit\'{e} Paris Sud.

\bibitem{PrRo}  C. Pr\'{e}v\^{o}t, M. R\"ockner (2007): A concise course on stochastic partial differential equations. Lecture Notes in Mathematics, 1905. Springer, Berlin.

\bibitem{Pr} J. Printems (2001): \emph{On the discretization in time of parabolic stochastic partial differential equations.}
Math. Mod. Numer. Anal. 35, 1055--1078.
\bibitem{Ro} 
M.~Romito.
\newblock Some probabilistic topics in the Navier--Stokes equations. 
\newblock Recent progress in the theory of the Euler and Navier--Stokes equations.
\newblock {\em London Math. Soc. Lecture Note Ser.} {\bf 430}, 175--232,  Cambridge Univ. Press, Cambridge, 2016.


\bibitem{ScoZha90}
L.~R. Scott and S.~Zhang, \emph{Finite element interpolation of nonsmooth
  functions satisfying boundary conditions}, Math. Comp. 54 (1990),
  no.~190, 483--493.



\bibitem{Ya1} Y. Yan (2004): Semidiscrete Galerkin Approximation for a Linear Stochastic Parabolic Partial Differential Equation Driven by an Additive Noise. Num. Math. 44, 829--847.

\bibitem{Ya2} Y. Yan (2005): Galerkin finite element methods for stochastic parabolic partial differential equations,
SIAM J. Numer. Anal., 43, pp. 1363--1384.

\bibitem{Yo} N. Yoshida (2012): Stochastic Shear thickening fluids: strong convergence of the Galerkin approximation and the energy inequality. The Annals of Applied Probability 2012, Vol. 22, No. 3, 1215--1242


\end{thebibliography}
\end{document}